\documentclass[11pt]{article}
\usepackage[square, numbers]{natbib}

\usepackage[a4paper,margin=1in]{geometry}
\usepackage{amsmath,amssymb,amsthm,mathtools}
\usepackage{enumitem}
\usepackage{hyperref}
\usepackage{mathrsfs}

\hypersetup{
    colorlinks=true,
    linkcolor=blue,
    citecolor=blue,
    urlcolor=blue
}

\newtheorem{theorem}{Theorem}[section]
\newtheorem{proposition}[theorem]{Proposition}
\newtheorem{corollary}[theorem]{Corollary}
\newtheorem{lemma}[theorem]{Lemma}
\theoremstyle{definition}
\newtheorem{definition}[theorem]{Definition}
\newtheorem{example}[theorem]{Example}
\theoremstyle{remark}
\newtheorem{remark}[theorem]{Remark}

\newcommand{\R}{\mathbb R}
\newcommand{\N}{\mathbb N}
\newcommand{\T}{\mathbb T}
\newcommand{\EE}{\mathcal E}

\newcommand{\norm}[1]{\left\|#1\right\|}

\newcommand{\zetaen}[2]{\zeta_{#1}^{(#2)}}

\newcommand{\Cn}{C([0,1])}

\title{Pathwise integration beyond Young \\ via Faber--Schauder energy spaces}
\author{
    \textsc{Donghan Kim} 
    \thanks{Department of Mathematical Sciences, KAIST, South Korea (E-mail: {\it kimdonghan@kaist.ac.kr}).\\
    Acknowledgements: The author is grateful to Ioannis Karatzas for valuable suggestions that improved the paper.}
}
\date{\today}

\begin{document}

\maketitle

\abstract{We develop a pathwise integration theory based on Faber--Schauder energy spaces. The approach replaces the classical H\"older--Young and finite-variation Young conditions by dyadic summability conditions expressed in terms of Faber--Schauder coefficients. On the normalized interval $[0,1]$, these conditions define Banach spaces $\EE^p$, which we call Faber--Schauder energy spaces. For $p,q>1$ satisfying $1/p+1/q\ge1$, we prove that every pair $f\in\EE^p$ and $g\in\EE^q$ admits a continuous pathwise integral $I_{f,g}$, constructed from dyadic left Riemann sums. We call $I_{f,g}$ the Faber--Schauder integral, and show that it depends boundedly and bilinearly on $(f,g)$ in the corresponding energy norms. The integral satisfies additivity, integration by parts, and a dyadic Young--Lo\`eve estimate. It is also the uniform limit of classical Riemann--Stieltjes integrals of finite Faber--Schauder approximations. The Faber--Schauder integral agrees with the classical Young integral whenever the latter is available, but also applies to deterministic and Gaussian examples for which neither the H\"older--Young condition nor the finite-variation Young condition can be verified. In this sense, it provides a Faber--Schauder coefficient-based extension of Young's framework.}

\bigskip

\section{Introduction}  \label{sec:introduction}

Young's classical theory \cite{Young1936} gives a pathwise meaning to the Riemann--Stieltjes integral
\[
    \int_0^t f\,dg
\]
under complementary regularity assumptions on the integrand and the integrator. In its H\"older form, Young's condition requires
\[
    f\in C^\alpha, \qquad g\in C^\beta, \qquad \alpha+\beta>1,
\]
and in its variation form, it requires finite $p$-variation for $f$ and $q$-variation for $g$ with
\[
    \frac1p+\frac1q>1.
\]
These conditions have become fundamental in pathwise integration, stochastic analysis, and rough path theory. They are, however, global regularity conditions: H\"older regularity is controlled by the worst normalized oscillation over all pairs of points $s\ne t$, while $p$-variation involves a supremum over all partitions of the interval; see \eqref{def:Holder-norm} and \eqref{def:p-variation} below. In this paper we ask whether a pathwise integral can instead be constructed from dyadic scale-by-scale information encoded in the Faber--Schauder coefficients. For notational simplicity, we work on the normalized interval $[0,1]$; the construction on a general compact interval is obtained by an affine change of time.

The starting point is the observation that the Faber--Schauder expansion offers a natural way of measuring the behavior of a path along dyadic levels. The celebrated result of Ciesielski~\cite{Ciesielski:isomorphism} characterizes H\"older regularity in terms of Faber--Schauder coefficients. More precisely, if
\[
    f(t)=f(0)+\big(f(1)-f(0)\big)t +\sum_{n=0}^{\infty}\sum_{k\in I_n}\theta^f_{n,k}e_{n,k}(t), \qquad t\in[0,1],
\]
is the Faber--Schauder expansion of $f\in C([0,1])$, then for $\alpha\in(0,1)$
\[
    f\in C^\alpha([0,1]) \qquad \text{if and only if} \qquad \sup_{n\ge0}\sup_{k\in I_n} 2^{n(\alpha-\frac12)}\big|\theta^f_{n,k}\big|<\infty.
\]
Here, the Faber--Schauder coefficients $\theta^f_{n,k}$ in \eqref{eq:FS-coeff} are normalized second-order oscillations of the path on dyadic intervals of length $2^{-n}$, and $e_{n,k}$ are the Faber--Schauder functions; see Section~\ref{subsec:FS-expansion}. Thus, the H\"older version of Young's condition already admits a Faber--Schauder coefficient interpretation.

A related coefficient phenomenon appears when one studies variation along a fixed partition sequence. In a recent paper \cite{das-kim2024}, the asymptotic behavior of the $p$-th variation along a fixed partition sequence was characterized in terms of the corresponding Faber--Schauder coefficients. In the dyadic case, this says that boundedness of the dyadic $p$-th variation is equivalent to boundedness of the level quantities
\begin{equation}    \label{equiv:zetaen}
    \sup_{n\ge0}\zetaen{f}{p}(n)<\infty, \qquad \zetaen{f}{p}(n):= \bigg(2^{-\frac{np}{2}}\sum_{k\in I_n}\big|\theta^f_{n,k}\big|^p \bigg)^{\frac1p}.
\end{equation}
Ciesielski's theorem and the dyadic $p$-th variation criterion \eqref{equiv:zetaen} therefore point to two distinct ways in which Faber--Schauder coefficients measure path regularity: H\"older regularity is captured by a uniform $\ell^\infty$ bound on individually rescaled coefficients, while dyadic $p$-th variation is captured by a level-wise $\ell^p$ aggregation of the rescaled coefficients. The present paper develops the latter viewpoint in a new direction. By requiring the level energies $\zetaen{f}{p}$ to be summable across dyadic levels, together with a Schauder tail condition, we obtain a coefficient-based framework for pathwise integration.

Faber-Schauder coefficients have also been used extensively in the study of stochastic sample-path regularity. Ciesielski, Kerkyacharian and Roynette \cite{CiesielskiKerkyacharianRoynette1993} developed constructive sequence-space descriptions of function spaces associated with Gaussian processes, including fractional Brownian motion, while Roynette \cite{Roynette1993} studied Brownian motion and stochastic integrals in Besov spaces using Faber-Schauder coefficients. Our construction is in the same coefficient-level spirit, but the conditions imposed below are designed to control the dyadic refinement defects of Riemann sums rather than to characterize Besov regularity alone.

This viewpoint is different from the usual $p$-variation framework in Young's theory. The $p$-variation seminorm of \eqref{def:p-variation} takes the supremum of $p$-th powered increment sums over all partitions of the interval, whereas we fix the canonical dyadic partition sequence associated with the Faber--Schauder system. Instead of controlling all possible partitions, we control how the dyadic Riemann sums change from one level to the next through summability of the level energies, and use a Schauder tail condition to pass from dyadic endpoints to continuous paths on the whole interval. In this sense, our construction may be viewed as a dyadic sewing procedure expressed in Faber--Schauder coordinates: the one-step errors created by refining the elementary increment
\[
    A_{s,t}:=f(s)\big(g(t)-g(s)\big)
\]
are controlled by products of the Faber--Schauder energy quantities.

This leads to the Faber--Schauder energy spaces $\EE^p$; see Definition~\ref{def:energy-spaces}. Membership in $\EE^p$ consists of two coefficient conditions. The first is an $\ell^p$-summability condition on level energies in \eqref{equiv:zetaen}. For conjugate exponents $p,q>1$, the conditions
\[
    \zetaen{f}{p}\in\ell^p, \qquad \zetaen{g}{q}\in\ell^q
\]
ensure the construction of the dyadic endpoint integral $\int_0^t f\,dg$ for every dyadic point $t$. The second is an $\ell^p$-summability condition on the Schauder tail sequence
\begin{equation}    \label{def:H_f-n}
    H_f(n):=\sum_{m=n}^{\infty} 2^{-\frac m2}\sup_{k\in I_m}|\theta_{m,k}^f|, \qquad n\ge0.
\end{equation}
This tail condition is the additional ingredient that promotes the dyadic endpoint construction to a continuous integral on the whole interval. More precisely, once the level-energy conditions on both $f$ and $g$ ensure the dyadic endpoint integral, the condition $H_f\in\ell^p$ gives its unique continuous extension to $[0,1]$. Thus, if $f\in\EE^p$ and $g\in\EE^q$, we obtain a well-defined continuous path
\[
    I_{f,g}(t)=\int_0^t f\,dg, \qquad t\in[0,1].
\]
We call $I_{f,g}$ the \emph{Faber--Schauder integral} of $f$ against $g$. It is bilinear and additive over intervals, satisfies the integration-by-parts identity, and admits a dyadic analogue of the Young--Lo\`eve estimate.

The energy spaces also provide a natural Banach-space setting for the integral. We prove that each $\EE^p$ carries a natural Banach norm and that the energy spaces are monotone in the exponent: if $1<p<q$, then $\EE^p\subset\EE^q$ continuously. This monotonicity allows the theory, first proved for conjugate exponents, to apply to every pair $p,q>1$ satisfying $1/p+1/q\ge1$. In particular, for such exponents, the map $(f,g)\mapsto I_{f,g}$ is bounded and bilinear from $\EE^p\times\EE^q$ into $C([0,1])$. We also show that dyadic piecewise linear functions, obtained by finite-level Faber--Schauder truncation, are dense in $\EE^p$, and that the Faber--Schauder coefficient map identifies $\EE^p$ isomorphically with a natural Banach sequence space. This coefficient-space representation is in the spirit of Ciesielski's isomorphism theorem~\cite{Ciesielski:isomorphism} for H\"older spaces. We further prove an additional regularity result showing that, under a stronger decay condition on the integrand, the Faber--Schauder integral inherits the energy regularity of the integrator, thereby allowing it to serve again as an integrator in the construction of iterated Faber--Schauder integrals.

Another important feature of the theory is its compatibility with classical bounded-variation integration at finite levels. Let $\Pi_N f$ and $\Pi_M g$ be the finite Faber--Schauder truncations of $f$ and $g$. These are continuous piecewise linear functions, so the classical Riemann--Stieltjes integral
\[
    \int_0^t \Pi_N f\,d(\Pi_M g)
\]
is well-defined. We show that
\[
    \sup_{t\in[0,1]} \left| \int_0^t \Pi_N f\,d(\Pi_M g)-I_{f,g}(t) \right| \xrightarrow{N,M\to\infty}0.
\]
Thus, the Faber--Schauder integral is not merely an abstract limit of dyadic left sums; it is also the uniform limit of classical Riemann--Stieltjes integrals of canonical finite dyadic piecewise linear approximations. These approximants are determined by a finite number of observations of the paths on dyadic grids.

The construction is compatible with the classical Young integral on the overlap of the two theories. In the H\"older--Young regime, H\"older regularity implies the required energy-space membership, and in the finite-variation formulation the two integrals coincide whenever the paths also satisfy the corresponding energy-space assumptions. We construct examples showing that the energy-space criterion is complementary to the classical Young criteria: the Faber--Schauder integral can be well defined even for pairs satisfying neither the H\"older--Young condition nor the finite-variation Young condition.

We also discuss fractional Brownian motions~(fBM). For a fBM $B^H$ with Hurst parameter $H\in(0,1)$, we identify the sharp energy-space threshold:
\[
    B^H\in\EE^p \quad\text{almost surely if and only if} \quad H>\frac1p.
\]
Consequently, for two fBMs with Hurst indices $H_1$ and $H_2$, the energy-space criterion for conjugate exponents exactly matches the familiar Young threshold $H_1+H_2>1$. We then combine fBM with random level-sparse Faber--Schauder processes to obtain mixed stochastic examples outside the H\"older--Young regime.

The paper is organized as follows. Section~\ref{sec:preliminaries} introduces the basic notation and preliminary results. Section~\ref{sec:main-results} constructs the Faber--Schauder integral and proves its main properties, including the integration-by-parts formula, a Young--Lo\`eve type estimate, the Banach-space structure of $\EE^p$, and approximation by classical Riemann--Stieltjes integrals. Section~\ref{sec:examples} provides deterministic and stochastic examples, including examples beyond the H\"older--Young and finite-variation Young regimes. Finally, Section~\ref{sec:conclusion} discusses the relation with F\"ollmer-type pathwise integration and outlines several directions for future research.

\bigskip

\section{Preliminaries}\label{sec:preliminaries}

In this section, we introduce the dyadic notation, the Faber--Schauder expansion, and the coefficient quantities used throughout the paper.

\medskip

\subsection{Dyadic partitions and dyadic left sums}

For each $n\in \N_0:=\N\cup\{0\}$, let
\begin{equation*}
    \T_n := \Big\{ t_{n,k}:=k2^{-n} : k=0,1,\dots,2^n \Big\}
\end{equation*}
be the level-$n$ dyadic partition of $[0,1]$, and set
\[
    \T := \bigcup_{n\ge0}\T_n.
\]
Thus, $\T$ is the dense set of dyadic points in $[0,1]$. Write the index set 
\[
    I_n:=\{0,1,\dots,2^n-1\}
\]
and denote the level-$n$ dyadic intervals by
\[
    I_{n,k}:=[t_{n,k},t_{n,k+1}], \qquad k\in I_n.
\]
We also denote the left and right children of $I_{n,k}$ by
\begin{equation}    \label{left-right-children}
    I_{n,k}^L:=I_{n+1,2k}, \qquad I_{n,k}^R:=I_{n+1,2k+1}.
\end{equation}
For a fixed dyadic point $t\in\T$, let
\[
    \ell(t):=\min\{m\ge0:\ t\in\T_m\}.
\]

\begin{definition}[Dyadic left sums at dyadic endpoints]
\label{def:dyadic-endpoint-sums}
    Fix $t\in\T$. For $f,g\in\Cn$, define the dyadic left Riemann sums over the level-$n$ dyadic points by
    \begin{equation}    \label{def:left-Riemann-sum}
        S_n^L(t;f,g):=\sum_{k=0}^{2^nt-1} f(t_{n,k})\big(g(t_{n,k+1})-g(t_{n,k})\big), \qquad n\ge\ell(t),
    \end{equation}
    with the convention that the empty sum is zero. For $t=1$, we simply write
    \begin{equation}    \label{def:left-Riemann-sum-1}
        S_n^L(f,g):=S_n^L(1;f,g)=\sum_{k\in I_n}f(t_{n,k})\big(g(t_{n,k+1})-g(t_{n,k})\big), \qquad n\ge0.
    \end{equation}
    We say that the \emph{dyadic integral} over $[0,t]$ exists if the sequence $\big(S_n^L(t;f,g)\big)_{n\ge\ell(t)}$ converges, and denote it by
    \begin{equation}    \label{def:dyadic-integral-on-T}
        \int_0^t f\,dg:=\lim_{n\to\infty}S_n^L(t;f,g).
    \end{equation}
\end{definition}

\medskip

\subsection{Faber--Schauder expansion}  \label{subsec:FS-expansion}

Let
\[
    \psi(t):=
    \begin{cases}
        ~1, & t\in [0,\frac12),\\
        -1, & t\in [\frac12,1),\\
        ~0, & \text{otherwise},
    \end{cases}
\]
and define the dyadic Haar functions \cite{haar1910} by
\begin{equation}    \label{Haar}
    \psi_{n,k}(t):=2^{n/2}\psi(2^nt-k), \qquad n\ge0,\ k\in I_n.
\end{equation}
 Let $\N_{-1}:=\{-1,0,1,\dots\}$. With
\[
    \psi_{-1,0}(t):=1 \quad \text{for } t\in[0,1], \qquad I_{-1}:=\{0\},
\]
the system $\{\psi_{n,k}:n\in\N_{-1},\ k\in I_n\}$ is the Haar orthonormal basis of $L^2([0,1])$. The associated Faber--Schauder functions \cite{faber1910, schauder} are
\[
    e_{n,k}(t):=\int_0^t\psi_{n,k}(s)\,ds, \qquad n \in \N_{-1},\ k\in I_n.
\]

Every $f\in\Cn$ admits a unique Faber--Schauder representation
\begin{equation}\label{eq:FS-full}
    f(t)=f(0)+\sum_{n=-1}^{\infty}\sum_{k\in I_n}\theta_{n,k}^f e_{n,k}(t), \qquad t\in[0,1].
\end{equation}
Here $\theta_{n,k}^f$ are the Faber--Schauder coefficients, given by
\begin{equation}\label{eq:FS-coeff-1}
    \theta_{-1,0}^f:=f(1)-f(0),
\end{equation}
and, for $n\ge0$ and $k\in I_n$,
\begin{equation}\label{eq:FS-coeff}
    \theta_{n,k}^f=2^{n/2}\Big(2f(t_{n+1,2k+1})-f(t_{n,k})-f(t_{n,k+1})\Big).
\end{equation}

Throughout the paper, we consider functions on the unit interval $[0,1]$. The same arguments can be adapted to functions on $[0,T]$ for any $T>0$ by using the corresponding Haar--Schauder system on $[0,T]$.

\medskip

\subsection{Level energies, Schauder tails, and dyadic kernels}

In this subsection, we collect several quantities defined from the Faber--Schauder coefficients $\theta_{n,k}^f$ of a path $f\in\Cn$, introduced in \eqref{eq:FS-coeff}. These quantities will appear frequently throughout the construction. The level-energy and Schauder-tail quantities below are the ones already used in \eqref{equiv:zetaen} and \eqref{def:H_f-n}; here we record them systematically.

\begin{definition}[Level energies]  \label{def:level-energies}
    Fix $p>1$. For $f\in\Cn$, define the level energy sequence $\zetaen{f}{p}=(\zetaen{f}{p}(n))_{n\in\N_{-1}}$ by
    \[
        \zetaen{f}{p}(-1):=2^{1-\frac1p}|\theta_{-1,0}^f|=2^{1-\frac1p}\big|f(1)-f(0)\big|,
    \]
    and, for $n\ge0$,
    \begin{equation}    \label{def:zetaen_f}
        \zetaen{f}{p}(n):=\bigg(2^{-\frac{np}{2}}\sum_{k\in I_n}|\theta_{n,k}^f|^p\bigg)^{\frac1p}.
    \end{equation}
\end{definition}

\begin{definition}[Schauder tail sequence]
\label{def:schauder-tail-sequence}
    For $f\in\Cn$, define
    \[
        H_f(n):=\sum_{m=n}^{\infty}2^{-\frac m2}\sup_{k\in I_m}|\theta_{m,k}^f|, \qquad n\ge0,
    \]
    whenever the series is finite. We allow $H_f(n)=+\infty$ otherwise.
\end{definition}

Note that the affine level ($n = -1$) is not included in $H_f$.

For $p>1$, define the geometric kernel
\[
    K_p(j):=2^{-j(1-\frac1p)}, \qquad j\ge0.
\]
Since $1-\frac1p>0$, we have for any $1 \le r < \infty$
\begin{equation}\label{eq:kernel-lr}
    \norm{K_p}_{\ell^r(\N_0)}=\bigg(\sum_{j=0}^{\infty}2^{-rj(1-\frac1p)}\bigg)^{\frac1r} = \bigg(\frac{1}{1-2^{-r(1-\frac1p)}}\bigg)^{\frac1r} < \infty.
\end{equation}

For a sequence $a=(a_m)_{m\in\N_{-1}}$, define the convolution with $K_p$ by
\[
    (K_p*a)_n:=\sum_{m=-1}^{n}K_p(n-m)a_m, \qquad n\ge0.
\]
In particular, for $f\in\Cn$ we set, using \eqref{def:zetaen_f},
\begin{equation}\label{eq:increment-control-sequence}
    \mathcal A^{(p)}_f(n):=(K_p*\zetaen{f}{p})_n=\sum_{m=-1}^{n}K_p(n-m)\zetaen{f}{p}(m), \qquad n\ge0.
\end{equation}

For later use, if $a=(a_n)_{n\ge0}$ and $r\ge1$, write the $\ell^r$-norm of the tail sequence
\[
    \Vert a\Vert_{\ell^r_{\ge N}}:=\bigg(\sum_{n=N}^{\infty} |a_n|^r\bigg)^{\frac1r}, \qquad N\ge0.
\]

Finally, we say that $p,q>1$ are \emph{conjugates} if
\[
    \frac1p+\frac1q=1.
\]

\medskip

\subsection{$p$-variation, dyadic $p$-th variation, and Faber--Schauder coefficients}   \label{subsec:p-var}

A key concept in Young's integration theory and rough path theory is the $p$-variation seminorm. For $f\in C([0,1])$ and $p\ge1$, define the $p$-variation of $f$ as
\begin{equation}    \label{def:p-variation}
    \|f\|_{p\text{-var}} := \bigg(\sup_{\pi\in\Pi([0,1])}\sum_{i=1}^{N}|f(t_i)-f(t_{i-1})|^p\bigg)^{\frac1p},
\end{equation}
where the supremum is taken over all partitions $\pi=(0=t_0<t_1<\cdots<t_N=1)$ of $[0,1]$. This quantity can be difficult to estimate directly, since it requires computing the sums of $p$-th powered increments over all partitions.

A related but different notion is the $p$-th variation along a fixed partition sequence. This distinction is important: a path may have infinite $p$-variation in the above supremum sense, while still having a finite limiting $p$-th variation along a fixed sequence of partitions. Brownian motion provides the standard example. Almost every Brownian path has infinite $2$-variation, $\Vert B \Vert_{2\text{-var}} = \infty$, whereas its quadratic variation along standard refining partition sequences, such as the dyadic partitions, equals the elapsed time, that is, $[B](t)=t$.

We use the following dyadic version of $p$-th variation along $(\T_n)_{n\ge0}$. For $f\in C([0,1])$ and $p\ge1$, define
\begin{equation}    \label{p-th-variation}
    [f]^{(p)}_{\T_n}(t):=\sum_{j\in I_n}\big|f(t_{n,j+1}\wedge t)-f(t_{n,j}\wedge t)\big|^p, \qquad t\in[0,1].
\end{equation}
If $[f]^{(p)}_{\T_n}$ converges uniformly to a continuous function on $[0,1]$, then we say that $f$ admits finite $p$-th variation along the dyadic partition sequence and write
\begin{equation*}
    [f]^{(p)}_{\T}(t):=\lim_{n\to\infty}[f]^{(p)}_{\T_n}(t), \qquad t\in[0,1].
\end{equation*}
In this case, the limit $t\mapsto [f]^{(p)}_{\T}(t)$ is nondecreasing; see \cite[Definition~1.1 and Lemma~1.3]{perkowski2019}.

For some examples in Section~\ref{sec:examples}, we will use the following Faber--Schauder coefficient criterion for the finiteness of dyadic $p$-th variation from a recent work \cite{das-kim2024}.

\begin{lemma}[Dyadic case of Theorem 4.3 of \cite{das-kim2024}]    \label{lem:dyadic-Xp-characterization}
    Let $p > 1$ and $f \in C([0, 1])$ have the Faber--Schauder representation \eqref{eq:FS-full}. Then we have the equivalent conditions:
    \begin{equation*}
        \limsup_{n \to \infty} \, [f]^{(p)}_{\T_n}(1) < \infty \qquad \iff \qquad \limsup_{n \to \infty} \, \zetaen{f}{p}(n) < \infty.
    \end{equation*}
\end{lemma}

It is clear from the definitions that
\[
    \limsup_{n\to\infty} \, [f]^{(p)}_{\T_n}(1) = \infty \qquad \Longrightarrow \qquad \Vert f \Vert_{p\text{-var}} = \infty.
\]
Hence the equivalence above implies
\begin{equation}    \label{con:equiv-zetaen-p-var}
    \limsup_{n \to \infty} \, \zetaen{f}{p}(n) = \infty \qquad \Longrightarrow \qquad \Vert f \Vert_{p\text{-var}} = \infty.
\end{equation}

\bigskip

\section{Construction and properties of the Faber--Schauder integral}   \label{sec:main-results}

\subsection{Dyadic integrals on the dyadic points}
\label{subsec:one-step-estimate}

In this subsection, we construct an integral defined on the dyadic points $\T$ as a limit of the left Riemann sum in Definition~\ref{def:dyadic-endpoint-sums}. A sufficient condition for the construction involves the Faber--Schauder coefficients of the integrand and the integrator. This integral will be extended from $\T$ to the whole interval $[0, 1]$ in the next subsection.

For $f\in\Cn$ and $k\in I_n$, define the left and right child increments of
$f$ over the dyadic interval $I_{n,k}=[t_{n,k},t_{n,k+1}]$ by
\[
    f_{n+1,L}(k):=f(t_{n+1,2k+1})-f(t_{n+1,2k}), \qquad
    f_{n+1,R}(k):=f(t_{n+1,2k+2})-f(t_{n+1,2k+1}).
\]
We define $g_{n+1,L}(k)$ and $g_{n+1,R}(k)$ analogously for $g\in\Cn$.

\begin{lemma}[One-step telescoping identity]
\label{lem:telescoping}
    For $f,g\in\Cn$, $n\ge0$, $j=0,\dots,2^n$, and recalling \eqref{def:left-Riemann-sum}, 
    \begin{equation*}
        S_{n+1}^L(t_{n,j};f,g)-S_n^L(t_{n,j};f,g)=\sum_{k=0}^{j-1}f_{n+1,L}(k)g_{n+1,R}(k).
    \end{equation*}
    In particular, taking $j=2^n$ and recalling \eqref{def:left-Riemann-sum-1} give
    \[
        S_{n+1}^L(f,g)-S_n^L(f,g)=\sum_{k\in I_n}f_{n+1,L}(k)g_{n+1,R}(k).
    \]
\end{lemma}

\begin{proof}
    For each $k=0,\dots,j-1$, a direct computation on the parent interval $I_{n,k}$ gives
    \begin{align*}
        &f(t_{n+1,2k})\big(g(t_{n+1,2k+1})-g(t_{n+1,2k})\big)+f(t_{n+1,2k+1})\big(g(t_{n+1,2k+2})-g(t_{n+1,2k+1})\big)\\
        & \hspace{6cm} -f(t_{n,k})\big(g(t_{n,k+1})-g(t_{n,k})\big)\\
        = \, &\big(f(t_{n+1,2k+1})-f(t_{n+1,2k})\big)\big(g(t_{n+1,2k+2})-g(t_{n+1,2k+1})\big).
    \end{align*}
    Summing over $k=0,\dots,j-1$ proves the result.
\end{proof}

For $0\le m\le n$ and $k\in I_n$, let $\kappa(m;n,k)\in I_m$ be the unique index such that
\[
    I_{n,k}\subset I_{m,\kappa(m;n,k)}.
\]
The Haar function $\psi_{m,\kappa(m;n,k)}$ in \eqref{Haar} is constant on $I_{n,k}$ when $m<n$, but not when $m=n$: in the latter case it changes sign at the midpoint of $I_{n,k}$. What we use is the following child-wise constancy: for each side $\sigma\in\{L,R\}$ there exists a sign
\[
    \varepsilon_m^\sigma(n,k)\in\{\pm1\}
\]
such that
\[
    \psi_{m,\kappa(m;n,k)}(u)=2^{\frac m2}\varepsilon_m^\sigma(n,k), \qquad \text{for a.e. }u\in I_{n,k}^\sigma.
\]
Here, recall from \eqref{left-right-children} the left $I_{n,k}^L$ and right $I_{n,k}^R$ child subintervals of $I_{n, k}$.

\begin{lemma}[Child increment estimate]
\label{lem:child-estimate}
    Let $r>1$ and $f\in\Cn$. Then, for every $n\ge0$, as in \eqref{eq:increment-control-sequence},
    \begin{equation*}
        \bigg(\sum_{k\in I_n}\big|f_{n+1,L}(k)\big|^r\bigg)^{\frac1r}\le \frac12\,\mathcal A^{(r)}_f(n).
    \end{equation*}
    The same estimate holds with $f_{n+1,L}$ replaced by $f_{n+1,R}$.
\end{lemma}

\begin{proof}
    Consider a left child increment across $I_{n,k}^L=I_{n+1,2k}$. Since the endpoints $t_{n+1,2k}$ and $t_{n+1,2k+1}$ are level-$(n+1)$ dyadic points, all Faber--Schauder levels $m\ge n+1$ vanish at both endpoints. Hence, only levels $m\le n$ contribute; for $0 \le m \le n$, by the child-wise sign above,
    \[
        e_{m,\kappa(m;n,k)}(t_{n+1,2k+1}) - e_{m,\kappa(m;n,k)}(t_{n+1,2k}) = 2^{-(n+1)}2^{\frac m2}\varepsilon_m^L(n,k).
    \]    
    Therefore, with the affine level $m=-1$, we obtain
    \[
        f_{n+1,L}(k)=2^{-(n+1)}\theta_{-1,0}^f+\sum_{m=0}^{n}2^{-(n+1)}2^{\frac m2}\varepsilon_m^L(n,k)\theta_{m,\kappa(m;n,k)}^f.
    \]
    By Minkowski's inequality,
    \[
        \bigg(\sum_{k\in I_n}\big|f_{n+1,L}(k)\big|^r\bigg)^{\frac1r}\le \bigg(\sum_{k\in I_n}\big|2^{-(n+1)}\theta_{-1,0}^f\big|^r\bigg)^{\frac1r}+\sum_{m=0}^{n}\bigg(\sum_{k\in I_n}\Big|2^{-(n+1)}2^{\frac m2}\theta_{m,\kappa(m;n,k)}^f\Big|^r\bigg)^{\frac1r}.
    \]
    From Definition~\ref{def:level-energies}, the affine-level contribution is
    \[
        \bigg(\sum_{k\in I_n}\big|2^{-(n+1)}\theta_{-1,0}^f\big|^r\bigg)^{\frac1r}=2^{-(n+1)}2^{\frac nr}\big|\theta_{-1,0}^f\big|=\frac12K_r(n+1)\zetaen{f}{r}(-1).
    \]
    This is precisely the reason for the normalization factor $2^{1-\frac1r}$ in the affine level: with this choice, the affine contribution appears as the $m=-1$ term in the same convolutional estimate as the levels $m\ge0$.
    
    For fixed $0\le m\le n$, every coefficient $\theta_{m,\ell}^f$ appears exactly $2^{n-m}$ times as $k$ ranges over $I_n$. Thus
    \begin{align*}
        \bigg(\sum_{k\in I_n}\Big|2^{-(n+1)}2^{\frac m2}\theta_{m,\kappa(m;n,k)}^f\Big|^r\bigg)^{\frac1r}
        &=\bigg(2^{n-m}2^{-r(n+1)}2^{\frac{mr}{2}}\sum_{\ell\in I_m}\big|\theta_{m,\ell}^f\big|^r\bigg)^{\frac1r}\\
        &=\frac12K_r(n-m)\zetaen{f}{r}(m).
    \end{align*}
    Summing over $m = -1, 0, \dots, n$ gives
    \[
        \bigg(\sum_{k\in I_n}\big|f_{n+1,L}(k)\big|^r\bigg)^{\frac1r}\le \frac12\sum_{m=-1}^{n}K_r(n-m)\zetaen{f}{r}(m)=\frac12\,\mathcal A^{(r)}_f(n).
    \]
    The proof for right-child increments is identical, using the signs $\varepsilon_m^R(n,k)$.
\end{proof}

\begin{proposition}[One-step dyadic estimate]   \label{prop:one-step-estimate}
    Let $p,q>1$ be conjugates, and let $f,g\in\Cn$. Then, for every $n\ge0$ and every $j=0,\dots,2^n$,
    \begin{equation*}
        \big|S_{n+1}^L(t_{n,j};f,g)-S_n^L(t_{n,j};f,g)\big|\le \frac14\,\mathcal A^{(p)}_f(n) \, \mathcal A^{(q)}_g(n).
    \end{equation*}
\end{proposition}

\begin{proof}
    By H\"older's inequality and Lemma~\ref{lem:telescoping},
    \[
        \big|S_{n+1}^L(t_{n,j};f,g)-S_n^L(t_{n,j};f,g)\big|\le \bigg(\sum_{k\in I_n}\big|f_{n+1,L}(k)\big|^p\bigg)^{\frac1p}\bigg(\sum_{k\in I_n}\big|g_{n+1,R}(k)\big|^q\bigg)^{\frac1q}.
    \]
    Applying Lemma~\ref{lem:child-estimate} to $f$ with exponent $p$ and to $g$ with exponent $q$ proves the estimate.
\end{proof}

\begin{remark}[Dyadic sewing viewpoint] \label{rem:dyadic-sewing-viewpoint}
    Proposition~\ref{prop:one-step-estimate} can be viewed as a dyadic sewing estimate. Set
    \[
        A^{f,g}_{s,t}:=f(s)\big(g(t)-g(s)\big), \qquad s,t\in\T, \qquad s\le t.
    \]
    Then the dyadic left sum of \eqref{def:left-Riemann-sum} can be written as
    \[
        S_n^L(t;f,g)=\sum_{k=0}^{2^nt-1}A^{f,g}_{t_{n,k},t_{n,k+1}}, \qquad t\in\T, \quad n\ge\ell(t).
    \]
    If $u=t_{n+1,2k+1}$ is the midpoint of the dyadic interval $[s,t]=[t_{n,k},t_{n,k+1}]$, then
    \[
        A^{f,g}_{s,u}+A^{f,g}_{u,t}-A^{f,g}_{s,t} = \big(f(u)-f(s)\big)\big(g(t)-g(u)\big).
    \]
    Hence, for $t\in\T_n$, the difference $S_{n+1}^L(t;f,g)-S_n^L(t;f,g)$ is obtained by summing these one-step refinement errors over the level-$n$ intervals contained in $[0,t]$. Proposition~\ref{prop:one-step-estimate} shows that this change under one dyadic refinement is controlled by the coefficient-side quantity
    \[
        \frac14\,\mathcal A_f^{(p)}(n) \, \mathcal A_g^{(q)}(n).
    \]
    This viewpoint is related in spirit to classical and stochastic sewing lemmas. In the classical sewing lemma one controls the three-point increment of a two-parameter map over arbitrary triples $s<u<t$, while stochastic sewing results replace deterministic estimates by moment and conditional-expectation bounds; see, for example, the fractional Brownian applications in \cite{MatsudaPerkowski2024}. The present construction is more rigid and deterministic: we sew only along the canonical dyadic refinement associated with the Faber--Schauder system, and the one-step refinement error is controlled directly by Faber--Schauder level energies, without using filtrations, adaptedness, or probabilistic cancellation.
\end{remark}

The next theorem states that if these one-step refinement errors are summable over the dyadic levels, then the dyadic integral of \eqref{def:dyadic-integral-on-T} is well-defined.

\begin{theorem}[Dyadic integrals under coefficient summability] \label{thm:coefficient-sum}
    Let $p,q>1$ be conjugates, and let $f,g\in\Cn$. If
    \begin{equation}\label{eq:coefficient-summability}
        \sum_{n=0}^{\infty}\mathcal A^{(p)}_f(n) \, \mathcal A^{(q)}_g(n) < \infty,
    \end{equation}
    then for every dyadic point $t\in\T$, the dyadic integral $\int_0^t f\,dg$ exists.
\end{theorem}

\begin{proof}
    Fix $t\in\T$. For $M>N\ge\ell(t)$,
    \[
        \big|S_M^L(t;f,g)-S_N^L(t;f,g)\big|\le \sum_{n=N}^{M-1}\big|S_{n+1}^L(t;f,g)-S_n^L(t;f,g)\big|.
    \]
    Since $t\in\T_n$ for every $n\ge\ell(t)$, Proposition~\ref{prop:one-step-estimate} gives
    \[
        \big|S_M^L(t;f,g)-S_N^L(t;f,g)\big|\le \frac14\sum_{n=N}^{M-1}\mathcal A^{(p)}_f(n) \, \mathcal A^{(q)}_g(n).
    \]
    By \eqref{eq:coefficient-summability}, the right-hand side tends to $0$ as $M, N \to \infty$. Hence, $\big(S_n^L(t;f,g)\big)_{n\ge\ell(t)}$ is Cauchy and therefore converges.
\end{proof}

\begin{corollary}[A separated coefficient criterion]
\label{cor:separated-coefficient-criterion}
    Let $p,q>1$ be conjugates, and let $f,g\in\Cn$. Recall Definition~\ref{def:level-energies} and suppose that
    \[
        \zetaen{f}{p}\in\ell^p(\N_{-1}), \qquad \zetaen{g}{q}\in\ell^q(\N_{-1}).
    \]
    Then for every dyadic point $t\in\T$, the dyadic integral $\int_0^t f\,dg$ exists.
\end{corollary}

\begin{proof}
    Define the shifted sequence $\widetilde\zeta_f^{p}=(\widetilde\zeta_f^{p}(j))_{j\ge0}$ by
    \[
        \widetilde\zeta_f^{p}(j):=\zetaen{f}{p}(j-1), \qquad j\ge0.
    \]
    Then $\widetilde\zeta_f^{p}\in\ell^p(\N_0)$ and, for every $n\ge0$,
    \[
        \mathcal A^{(p)}_f(n)=(K_p*\widetilde\zeta_f^{p})_{n+1}.
    \]
    Hence, by Young's convolution inequality and \eqref{eq:kernel-lr},
    \begin{equation}    \label{ineq:Ap-bound}
        \Vert\mathcal A^{(p)}_f\Vert_{\ell^p(\N_0)}
        \le \Vert K_p*\widetilde\zeta_f^{p} \Vert_{\ell^p(\N_0)}
        \le \Vert K_p\Vert_{\ell^1(\N_0)} \Vert\widetilde\zeta_f^{p}\Vert_{\ell^p(\N_0)}
        = \Vert K_p\Vert_{\ell^1(\N_0)} \Vert\zetaen{f}{p}\Vert_{\ell^p(\N_{-1})}.
    \end{equation}
    Thus $\mathcal A^{(p)}_f\in\ell^p(\N_0)$. Similarly, $\mathcal A^{(q)}_g\in\ell^q(\N_0)$. Therefore, by H\"older's inequality,
    \begin{equation}    \label{ineq:A_pA_q-l1}
        \sum_{n=0}^{\infty} \mathcal A^{(p)}_f(n) \, \mathcal A^{(q)}_g(n)
        \le \Vert\mathcal A^{(p)}_f\Vert_{\ell^p(\N_0)} \, \Vert\mathcal A^{(q)}_g\Vert_{\ell^q(\N_0)} < \infty.
    \end{equation}
    The conclusion follows from Theorem~\ref{thm:coefficient-sum}.
\end{proof}

Whenever the conclusion of Theorem~\ref{thm:coefficient-sum} holds for a pair $(f,g)$, we write
\[
    I_{f,g}^{\T}(t):=\int_0^t f\,dg, \qquad t\in\T.
\]

\medskip

\subsection{Continuous extension to the whole interval}
\label{subsec:continuous-extension}

By Corollary~\ref{cor:separated-coefficient-criterion}, if $p,q>1$ are conjugates and
\[
    \zetaen{f}{p}\in\ell^p(\N_{-1}), \qquad \zetaen{g}{q}\in\ell^q(\N_{-1}),
\]
then the dyadic integral
\[
    I_{f,g}^{\T}(t):=\int_0^t f\,dg, \qquad t\in\T,
\]
is well-defined on the dyadic set $\T$. Moreover, the inequality \eqref{ineq:A_pA_q-l1} in the proof shows $\mathcal A^{(p)}_f \mathcal A^{(q)}_g \in \ell^1(\N_0)$.

The next result gives a sufficient condition under which this dyadic integral map is uniformly continuous on $\T$, and hence admits a unique continuous extension to $[0,1]$. We recall the level energies and Schauder tail sequence in Definitions~\ref{def:level-energies}, \ref{def:schauder-tail-sequence}.

\begin{theorem}[Continuous extension under a Schauder tail condition]   \label{thm:uniform-continuity}
    Let $p,q>1$ be conjugates, and let $f,g\in\Cn$. Suppose that $\zetaen{f}{p}\in\ell^p(\N_{-1})$, $\zetaen{g}{q}\in\ell^q(\N_{-1})$, and
    \begin{equation}\label{eq:Hf-condition}
        H_f\in\ell^p(\N_0).
    \end{equation}
    Then the mapping
    \[
        I_{f,g}^{\T}:\T\to\R, \qquad I_{f,g}^{\T}(t)=\int_0^t f\,dg,
    \]
    of \eqref{def:dyadic-integral-on-T} is uniformly continuous on $\T$. Thus, it admits a unique continuous extension to $[0,1]$.
\end{theorem}

The role of the Schauder tail condition \eqref{eq:Hf-condition} is to control the oscillation of the integrand inside small dyadic intervals, thereby allowing the endpoint estimates to pass to a continuous extension on $[0,1]$. At this stage, the tail condition is required only for the integrand in the extension of $I_{f,g}^{\T}$; the corresponding symmetry is recovered later through integration by parts, as explained in Remark~\ref{rem:one-sided-tail-condition} below. To prove Theorem~\ref{thm:uniform-continuity}, we use the following lemmas.

\begin{lemma}[Refinement remainders]    \label{lem:refinement-remainders}
    Let $p,q>1$ be conjugates, and let $f,g\in\Cn$. Suppose that $\zetaen{f}{p}\in\ell^p(\N_{-1})$ and $\zetaen{g}{q}\in\ell^q(\N_{-1})$. Then the following estimates hold.

    \smallskip

    \noindent\emph{(i)} For every $N\ge0$ and $j=0,\dots,2^N-1$, there exists a remainder $R_{N,j}$ such that
    \[
        I_{f,g}^{\T}(t_{N,j+1})-I_{f,g}^{\T}(t_{N,j})
        = f(t_{N,j})\big(g(t_{N,j+1})-g(t_{N,j})\big)+R_{N,j},
    \]
    and
    \begin{equation}\label{eq:grid-remainder-bound}
        |R_{N,j}|\le\frac14 \big\Vert \mathcal A^{(p)}_f \, \mathcal A^{(q)}_g \big\Vert_{\ell^1_{\ge N}}.
    \end{equation}

    \smallskip

    \noindent\emph{(ii)} Fix $t\in\T$ and $N\ge0$. For $n\ge N$, set
    \[
        t_n:=2^{-n}\lfloor2^nt\rfloor.
    \]
    Then there exist remainders $R_n(t)$ such that
    \[
        I_{f,g}^{\T}(t_{n+1})-I_{f,g}^{\T}(t_n)
        = f(t_n)\big(g(t_{n+1})-g(t_n)\big)+R_n(t),
    \]
    and
    \begin{equation}\label{eq:approach-remainder-bound}
        \sum_{n=N}^{\infty}\big|R_n(t)\big| \le \frac14 \big\Vert \mathcal A^{(p)}_f \, \mathcal A^{(q)}_g \big\Vert_{\ell^1_{\ge N}}.
    \end{equation}
\end{lemma}

\begin{proof}
    We first prove \eqref{eq:grid-remainder-bound}. Fix $N\ge0$ and $j=0,\dots,2^N-1$. Since $t_{N,j},t_{N,j+1}\in\T_N$, the level-$N$ left sum over $[t_{N,j},t_{N,j+1}]$ is exactly $f(t_{N,j})\big(g(t_{N,j+1})-g(t_{N,j})\big)$. By telescoping the refinements from level $N$ to infinity,
    \[
        I_{f,g}^{\T}(t_{N,j+1})-I_{f,g}^{\T}(t_{N,j})
        = f(t_{N,j})\big(g(t_{N,j+1})-g(t_{N,j})\big)+\sum_{m=N}^{\infty}\Delta_m(t_{N,j},t_{N,j+1}),
    \]
    where
    \[
        \Delta_m(a,b):=\big(S_{m+1}^L(b;f,g)-S_m^L(b;f,g)\big)-\big(S_{m+1}^L(a;f,g)-S_m^L(a;f,g)\big).
    \]
    By Lemma~\ref{lem:telescoping},
    \[
        \Delta_m(t_{N,j},t_{N,j+1}) = \sum_{k:\,I_{m,k}\subset[t_{N,j},t_{N,j+1}]} f_{m+1,L}(k)g_{m+1,R}(k).
    \]
    Hence, by H\"older's inequality and Lemma~\ref{lem:child-estimate},
    \[
        \big|\Delta_m(t_{N,j},t_{N,j+1})\big|
        \le \sum_{k\in I_m}\big|f_{m+1,L}(k)\big|\big|g_{m+1,R}(k)\big|
        \le \frac14 \mathcal A^{(p)}_f(m) \, \mathcal A^{(q)}_g(m).
    \]
    As in the proof of Corollary~\ref{cor:separated-coefficient-criterion}, the assumptions $\zetaen{f}{p}\in\ell^p(\N_{-1})$ and $\zetaen{g}{q}\in\ell^q(\N_{-1})$ imply $\mathcal A^{(p)}_f \mathcal A^{(q)}_g \in\ell^1(\N_0)$, therefore the series
    \[
        R_{N, j} := \sum_{m=N}^{\infty}\Delta_m(t_{N,j},t_{N,j+1})
    \]
    is absolutely convergent and \eqref{eq:grid-remainder-bound} follows.

    We next prove \eqref{eq:approach-remainder-bound}. Fix $t\in\T$ and define $t_n=2^{-n}\lfloor2^nt\rfloor$. If $t_{n+1}=t_n$, then the corresponding increment is zero, and we set $R_n(t)=0$. If $t_{n+1}>t_n$, then $[t_n,t_{n+1}]$ is one level-$(n+1)$ dyadic interval. Telescoping from level $n+1$ to infinity gives
    \[
        I_{f,g}^{\T}(t_{n+1})-I_{f,g}^{\T}(t_n)
        = f(t_n)\big(g(t_{n+1})-g(t_n)\big)+\sum_{m=n+1}^{\infty}\Delta_{m,n}(t),
    \]
    where
    \[
        \Delta_{m,n}(t):=\big(S_{m+1}^L(t_{n+1};f,g)-S_m^L(t_{n+1};f,g)\big)-\big(S_{m+1}^L(t_n;f,g)-S_m^L(t_n;f,g)\big).
    \]
    We define
    \[
        R_n(t):=\sum_{m=n+1}^{\infty}\Delta_{m,n}(t).
    \]
    For $m\ge n+1$, Lemma~\ref{lem:telescoping} gives
    \[
        \Delta_{m,n}(t)=\sum_{k:\,I_{m,k}\subset[t_n,t_{n+1}]}f_{m+1,L}(k)g_{m+1,R}(k).
    \]
    For fixed $m$, the intervals $[t_n,t_{n+1}]$, $n=N,\dots,m-1$, are pairwise disjoint up to endpoints. Therefore the corresponding index sets $\{k\in I_m:I_{m,k}\subset[t_n,t_{n+1}]\}$ are disjoint. Hence, by H\"older's inequality and Lemma~\ref{lem:child-estimate},
    \[
        \sum_{n=N}^{m-1}\big|\Delta_{m,n}(t)\big|
        \le \sum_{k\in I_m}\big|f_{m+1,L}(k)\big|\big|g_{m+1,R}(k)\big|
        \le \frac14 \mathcal A^{(p)}_f(m) \, \mathcal A^{(q)}_g(m).
    \]
    Tonelli's theorem yields
    \[
        \sum_{n=N}^{\infty}\big|R_n(t)\big|
        \le \sum_{m=N}^{\infty}\sum_{n=N}^{m-1}\big|\Delta_{m,n}(t)\big|
        \le \frac14\sum_{m=N}^{\infty} \mathcal A^{(p)}_f(m) \, \mathcal A^{(q)}_g(m) = \frac14 \big\Vert \mathcal A^{(p)}_f \, \mathcal A^{(q)}_g \big\Vert_{\ell^1_{\ge N}}.
    \]
\end{proof}

In what follows, we write $\omega_h$ for the modulus of continuity of $h\in\Cn$, namely
\[
    \omega_h(\delta):=\sup_{\substack{u,v\in[0,1]\\ |u-v|\le\delta}}\big|h(u)-h(v)\big|, \qquad \delta>0.
\]

\begin{lemma}[Dyadic estimates for the integrand]
\label{lem:dyadic-approach-f}
    Let $r>1$ and $h\in\Cn$. Suppose that $\zetaen{h}{r}\in\ell^r(\N_{-1})$ and $H_h\in\ell^r(\N_0)$. Then, for every $N\ge0$ and every $u\in[t_{N,k},t_{N,k+1}]$,
    \begin{equation}\label{ineq:two-endpoint-estimates}
        \big|h(u)-h(t_{N,k})\big| \le \mathcal A^{(r)}_h(N)+\frac12H_h(N), \qquad \big|h(u)-h(t_{N,k+1})\big| \le \mathcal A^{(r)}_h(N)+\frac12H_h(N).
    \end{equation}
    Consequently, 
    \begin{equation}    \label{ineq:oscillation-estimate}
        \omega_h(2^{-N}) \le 2\mathcal A^{(r)}_h(N)+H_h(N).
    \end{equation}

    Moreover, for $t\in\T\cap[0,1)$ and $n\ge0$, set $t_n:=2^{-n}\lfloor2^nt\rfloor$. Then, for every $N\ge0$,
    \begin{equation}\label{eq:dyadic-approach-f-estimate}
        \sup_{t\in\T\cap[0,1)} \bigg(\sum_{n=N}^{\infty}\big|h(t)-h(t_n)\big|^r\bigg)^{\frac1r}
        \le \Vert\mathcal A^{(r)}_h\Vert_{\ell^r_{\ge N}} + \frac12\Vert H_h\Vert_{\ell^r_{\ge N}}.
    \end{equation}
\end{lemma}

\begin{proof}
    Fix $N\ge0$ and $u\in[t_{N,k},t_{N,k+1}]$. Consider the level-$N$ dyadic interpolation
    \[
        \Pi_Nh(v):=h(0)+\sum_{m=-1}^{N-1}\sum_{\ell\in I_m}\theta_{m,\ell}^h e_{m,\ell}(v), \qquad v\in[0,1].
    \]
    Since the Schauder tail vanishes at level-$N$ dyadic points,
    \[
        \Pi_Nh(t_{N,k})=h(t_{N,k}), \qquad \Pi_Nh(t_{N,k+1})=h(t_{N,k+1}).
    \]
    Moreover, $\Pi_Nh$ is affine on $[t_{N,k},t_{N,k+1}]$. Hence
    \[
        \big|\Pi_Nh(u)-\Pi_Nh(t_{N,k})\big|
        \le \big|\Pi_Nh(t_{N, k+1})-\Pi_Nh(t_{N,k})\big| = \big|h(t_{N,k+1})-h(t_{N,k})\big|,
    \]
    and the same estimate holds with $t_{N,k}$ replaced by $t_{N,k+1}$. The increment $h(t_{N,k+1})-h(t_{N,k})$ is the sum of the two child increments over the level-$N$ interval, so Lemma~\ref{lem:child-estimate} gives
    \[
        \big|h(t_{N,k+1})-h(t_{N,k})\big| \le \mathcal A^{(r)}_h(N).
    \]
    For the tail term, at each level $m\ge N$, at most one Faber--Schauder tent is nonzero at $u$, and its height is at most $2^{-\frac m2-1}$. Therefore
    \begin{equation}    \label{ineq:N-dyadic-interpol}
        \big|h(u)-\Pi_Nh(u)\big| = \bigg|\sum_{m=N}^{\infty}\sum_{\ell \in I_m} \theta_{m,\ell}^h \, e_{m,\ell}(u)\bigg|
        \le \frac12\sum_{m=N}^{\infty}2^{-\frac m2}\sup_{\ell\in I_m}\big|\theta_{m,\ell}^h\big|
        = \frac12H_h(N).
    \end{equation}
    Combining these estimates proves the two estimates \eqref{ineq:two-endpoint-estimates}.

    If $|u-v|\le2^{-N}$, then $u$ and $v$ lie either in the same level-$N$ dyadic interval or in two adjacent level-$N$ dyadic intervals. In both cases, there is a level-$N$ dyadic point $z$ such that $u$ and $v$ each lie in a level-$N$ interval having $z$ as one endpoint. Hence
    \[
        \big|h(u)-h(v)\big| \le 2\mathcal A^{(r)}_h(N)+H_h(N).
    \]
    Taking the supremum over all such $u,v$ proves the oscillation estimate \eqref{ineq:oscillation-estimate}.

    Now fix $t\in\T\cap[0,1)$ and $n\ge0$. Applying the first endpoint estimate at level $n$ to the interval containing $t$ gives
    \[
        \big|h(t)-h(t_n)\big| \le \mathcal A^{(r)}_h(n)+\frac12H_h(n).
    \]
    Taking the $\ell^r$ norm over $n\ge N$ and then the supremum over $t\in\T\cap[0,1)$ proves \eqref{eq:dyadic-approach-f-estimate}.
\end{proof}

\begin{lemma}[Uniform dyadic approximation of the integral map]
\label{lem:uniform-integral-map-approx}
    Let $p,q>1$ be conjugates, and let $f,g\in\Cn$. Suppose that $\zetaen{f}{p}\in\ell^p(\N_{-1})$, $\zetaen{g}{q}\in\ell^q(\N_{-1})$, and $H_f\in\ell^p(\N_0)$. Define
    \[
        F(t):=I_{f,g}^{\T}(t), \qquad t\in\T,
    \]
    and, for $N\ge0$,
    \[
        t_N:=2^{-N}\lfloor2^Nt\rfloor, \qquad F_N(t):=F(t_N), \qquad t\in\T.
    \]
    Then $F_N\to F$ uniformly on $\T$, i.e.,
    \[
        \sup_{t\in\T}\big|F(t)-F_N(t)\big|\xrightarrow{N\to\infty}0.
    \]
\end{lemma}

\begin{proof}
    The claim is trivial at $t=1$, since $t_N=1$ for all $N$. Fix $t\in\T\cap[0,1)$. Since $t_n=2^{-n}\lfloor2^nt\rfloor$ satisfies $t_n\uparrow t$ and $t_n=t$ for all sufficiently large $n$, we have
    \begin{equation*}
        F(t)-F_N(t)=\sum_{n=N}^{\infty}\big(F(t_{n+1})-F(t_n)\big).
    \end{equation*}
    By Lemma~\ref{lem:refinement-remainders} (ii),
    \[
        F(t_{n+1})-F(t_n)=f(t_n)\big(g(t_{n+1})-g(t_n)\big)+R_n(t),
    \]
    with
    \[
        \sum_{n=N}^{\infty} \big|R_n(t)\big| \le \frac14 \big\Vert \mathcal A^{(p)}_f \, \mathcal A^{(q)}_g \big\Vert_{\ell^1_{\ge N}}.
    \]
    Therefore
    \[
        F(t)-F(t_N)-f(t)\big(g(t)-g(t_N)\big) = \sum_{n=N}^{\infty}\big(f(t_n)-f(t)\big)\big(g(t_{n+1})-g(t_n)\big)+\sum_{n=N}^{\infty}R_n(t).
    \]
    Each nonzero increment $g(t_{n+1})-g(t_n)$ is a left-child increment of $g$ at level $n+1$. Hence Lemma~\ref{lem:child-estimate} gives
    \[
        \big|g(t_{n+1})-g(t_n)\big|\le \frac12\mathcal A^{(q)}_g(n).
    \]
    By H\"older's inequality and Lemma~\ref{lem:dyadic-approach-f},
    \[
        \sum_{n=N}^{\infty}\big|f(t_n)-f(t)\big|\big|g(t_{n+1})-g(t_n)\big|\le \frac12\Vert\mathcal A^{(q)}_g\Vert_{\ell^q_{\ge N}}\bigg(\Vert\mathcal A^{(p)}_f\Vert_{\ell^p_{\ge N}}+\frac12\Vert H_f\Vert_{\ell^p_{\ge N}}\bigg).
    \]
    Combining the last two estimates proves for every $N\ge0$ and every $t\in\T$,
    \begin{align}
        \Big|F(t)-F(t_N)&-f(t)\big(g(t)-g(t_N)\big)\Big| \label{eq:projection-bound-on-T}
        \\
        &\le \frac12\Vert\mathcal A^{(q)}_g\Vert_{\ell^q_{\ge N}}\bigg(\Vert\mathcal A^{(p)}_f\Vert_{\ell^p_{\ge N}}+\frac12\Vert H_f\Vert_{\ell^p_{\ge N}}\bigg)+\frac14 \big\Vert \mathcal A^{(p)}_f \, \mathcal A^{(q)}_g \big\Vert_{\ell^1_{\ge N}}.    \nonumber
    \end{align}
    Finally,
    \[
        |F(t)-F_N(t)|\le \big|F(t)-F(t_N)-f(t)\big(g(t)-g(t_N)\big)\big|+\|f\|_{\infty}\omega_g(2^{-N}).
    \]
    Taking the supremum over $t\in\T$ gives
    \begin{align*}
        \sup_{t\in\T}\big|F(t)&-F_N(t)\big|
        \\
        &\le \frac12\Vert \mathcal A^{(q)}_g \Vert_{\ell^q_{\ge N}}\bigg(\Vert \mathcal A^{(p)}_f \Vert_{\ell^p_{\ge N}}+\frac12\Vert H_f\Vert_{\ell^p_{\ge N}}\bigg) + \frac14 \big\Vert \mathcal A^{(p)}_f \, \mathcal A^{(q)}_g \big\Vert_{\ell^1_{\ge N}} + \|f\|_{\infty}\omega_g(2^{-N}).
    \end{align*}
    The right-hand side converges to $0$ as $N\to\infty$, since $\mathcal A^{(p)}_f\in\ell^p$, $\mathcal A^{(q)}_g \in\ell^q$, $H_f\in\ell^p$, $\mathcal A^{(p)}_f \mathcal A^{(q)}_g\in\ell^1$, and $g$ is uniformly continuous.    
\end{proof}

\begin{lemma}[Vanishing jumps on dyadic grids]
\label{lem:vanishing-grid-jumps}
    Let $p,q>1$ be conjugates, and let $f,g\in\Cn$. Suppose that $\zetaen{f}{p}\in\ell^p(\N_{-1})$ and $\zetaen{g}{q}\in\ell^q(\N_{-1})$. Recall $F(t):=I_{f,g}^{\T}(t)$ for $t\in\T$, then 
    \[
        \max_{0\le j<2^N}\big|F(t_{N,j+1})-F(t_{N,j})\big|\xrightarrow{N\to\infty}0.
    \]
\end{lemma}

\begin{proof}
    By Lemma~\ref{lem:refinement-remainders} (i),
    \[
        F(t_{N,j+1})-F(t_{N,j})=f(t_{N,j})\big(g(t_{N,j+1})-g(t_{N,j})\big)+R_{N,j},
    \]
    where
    \[
        |R_{N,j}|\le\frac14 \big\Vert \mathcal A^{(p)}_f \, \mathcal A^{(q)}_g \big\Vert_{\ell^1_{\ge N}}.
    \]
    The first term is bounded by $\|f\|_{\infty}\omega_g(2^{-N})$. Taking the maximum over $j$ gives 
    \begin{equation*}
        \max_{0\le j<2^N}\big|F(t_{N,j+1})-F(t_{N,j})\big|\le\|f\|_{\infty}\omega_g(2^{-N})+\frac14 \big\Vert \mathcal A^{(p)}_f \, \mathcal A^{(q)}_g \big\Vert_{\ell^1_{\ge N}}.
    \end{equation*}
    The right-hand side tends to zero because $g$ is uniformly continuous and $\mathcal A^{(p)}_f \mathcal A^{(q)}_g\in\ell^1$.
\end{proof}

We are now ready to prove Theorem~\ref{thm:uniform-continuity}.

\begin{proof}[Proof of Theorem~\ref{thm:uniform-continuity}]
    For $t\in\T$, set
    \[
        F(t):=I_{f,g}^{\T}(t), \qquad F_N(t):=F(2^{-N}\lfloor2^Nt\rfloor).
    \]
    By Lemma~\ref{lem:uniform-integral-map-approx},
    \[
        \sup_{u\in\T}\big|F(u)-F_N(u)\big|\xrightarrow{N\to\infty}0.
    \]
    By Lemma~\ref{lem:vanishing-grid-jumps},
    \[
        \max_{0\le j<2^N}\big|F(t_{N,j+1})-F(t_{N,j})\big|\xrightarrow{N\to\infty}0.
    \]
    Let $\varepsilon>0$. Choose $N$ large enough so that
    \[
        2\sup_{u\in\T}\big|F(u)-F_N(u)\big|+\max_{0\le j<2^N}\big|F(t_{N,j+1})-F(t_{N,j})\big|<\varepsilon.
    \]
    If $s,t\in\T$ and $|s-t|<2^{-N}$, then
    \[
        s_N:=2^{-N}\lfloor2^Ns\rfloor, \qquad t_N:=2^{-N}\lfloor2^Nt\rfloor
    \]
    are either equal or adjacent level-$N$ dyadic points. Hence
    \[
        \big|F_N(s)-F_N(t)\big|=\big|F(s_N)-F(t_N)\big|\le\max_{0\le j<2^N}\big|F(t_{N,j+1})-F(t_{N,j})\big|.
    \]
    Therefore
    \[
        \big|F(s)-F(t)\big|\le\big|F(s)-F_N(s)\big|+\big|F_N(s)-F_N(t)\big|+\big|F_N(t)-F(t)\big|<\varepsilon.
    \]
    This proves uniform continuity of $F=I_{f,g}^{\T}$ on $\T$. Since $\T$ is dense in $[0,1]$ and $\R$ is complete, the standard extension theorem for uniformly continuous maps on dense subsets gives a unique continuous extension of $F$ to $[0,1]$.
\end{proof}

Whenever the assumptions of Theorem~\ref{thm:uniform-continuity} hold, we can define the dyadic integral on the whole interval $[0, 1]$.

\begin{definition}[Faber--Schauder integral on the unit interval]   \label{def:FS-integral-unit-interval}
    Let $f,g\in\Cn$. Suppose that the dyadic integral $I_{f,g}^{\T}$ is well-defined on $\T$ and admits a unique continuous extension to $[0,1]$. We denote this extension by
    \[
        I_{f,g}:[0,1]\to\R
    \]
    and say that the \emph{Faber--Schauder integral} exists on $[0,1]$.
\end{definition}

We conclude this subsection with the following corollary, which provides a simpler sufficient condition for the Schauder tail assumption \eqref{eq:Hf-condition}.

\begin{corollary}[A weighted sufficient condition for continuous extension]
\label{cor:weighted-boundary-tail}
    Let $p,q>1$ be conjugates, and let $f,g\in\Cn$. Suppose that $\zetaen{f}{p}\in\ell^p(\N_{-1})$, $\zetaen{g}{q}\in\ell^q(\N_{-1})$, and
    \begin{equation}\label{eq:weighted-tail-condition}
        \sum_{m=1}^{\infty}m^{\frac1p}2^{-\frac m2}\sup_{k\in I_m}\big|\theta_{m,k}^f\big|<\infty.
    \end{equation}
    Then $H_f\in\ell^p(\N_0)$. Consequently, the unique continuous extension $I_{f, g}$ of $I_{f,g}^{\T}$ exists on $[0, 1]$.
\end{corollary}

\begin{proof}
    By Minkowski's inequality,
    \begin{align*}
        \Vert H_f\Vert_{\ell^p}
        &=
        \bigg\Vert\sum_{m=0}^{\infty} 2^{-\frac m2}\sup_{k\in I_m}\big|\theta_{m,k}^f\big| \mathbf 1_{\{0,\dots,m\}}\bigg\Vert_{\ell^p}
        \\
        &\le
        \sum_{m=0}^{\infty} \Big( 2^{-\frac m2}\sup_{k\in I_m}\big|\theta_{m,k}^f\big| \Big) \big\Vert\mathbf 1_{\{0,\dots,m\}}\big\Vert_{\ell^p}
        = \sum_{m=0}^{\infty}(m+1)^{\frac1p} 2^{-\frac m2}\sup_{k\in I_m}\big|\theta_{m,k}^f\big|.
    \end{align*}
    The $m=0$ term is finite, and $(m+1)^{\frac1p}\le2^{\frac1p}m^{\frac1p}$ holds for $m\ge1$. Hence \eqref{eq:weighted-tail-condition} implies $\Vert H_f\Vert_{\ell^p}<\infty$. The conclusion follows from Theorem~\ref{thm:uniform-continuity}.
\end{proof}

\medskip

\subsection{Algebraic properties and integration by parts}
\label{subsec:algebraic-properties}

We now provide basic algebraic properties of the Faber--Schauder integrals on the unit interval. Throughout this subsection, the statement that $I_{f,g}$ exists on $[0,1]$ is understood in the sense of Definition~\ref{def:FS-integral-unit-interval}. Whenever $I_{f,g}$ exists on $[0,1]$, we define the interval integral by
\[
    \int_s^t f\,dg:=I_{f,g}(t)-I_{f,g}(s).
\]

\begin{proposition}[Bilinearity and additivity]
\label{prop:extended-bilinearity-additivity}
    Let $f_1,f_2,g\in\Cn$ and $\alpha,\beta\in\R$. If the Faber--Schauder integrals $I_{f_1,g}$ and $I_{f_2,g}$ exist on $[0,1]$, then $I_{\alpha f_1+\beta f_2,g}$ also exists on $[0,1]$, and for every $t\in[0,1]$,
    \[
        I_{\alpha f_1+\beta f_2,g}(t)=\alpha I_{f_1,g}(t)+\beta I_{f_2,g}(t).
    \]
    Consequently, for every $s,t\in[0,1]$,
    \[
        \int_s^t(\alpha f_1+\beta f_2)\,dg
        = \alpha\int_s^t f_1\,dg+\beta\int_s^t f_2\,dg.
    \]

    Similarly, let $f,g_1,g_2\in\Cn$ and $\alpha,\beta\in\R$. If $I_{f,g_1}$ and $I_{f,g_2}$ exist on $[0,1]$, then $I_{f,\alpha g_1+\beta g_2}$ also exists on $[0,1]$, and for every $t\in[0,1]$
    \[
        I_{f,\alpha g_1+\beta g_2}(t)=\alpha I_{f,g_1}(t)+\beta I_{f,g_2}(t),
    \]
    and consequently, for every $s,t\in[0,1]$,
    \[
        \int_s^t f\,d(\alpha g_1+\beta g_2)
        = \alpha\int_s^t f\,dg_1+\beta\int_s^t f\,dg_2.
    \]

    Finally, if $I_{f,g}$ exists on $[0,1]$, then
    \[
        \int_r^t f\,dg=\int_r^s f\,dg+\int_s^t f\,dg, \qquad r,s,t\in[0,1].
    \]
    In particular,
    \[
        \int_s^s f\,dg=0, \qquad \int_s^t f\,dg=-\int_t^s f\,dg.
    \]
\end{proposition}

\begin{proof}
    For $t\in\T$, bilinearity of the finite dyadic left sums gives
    \[
        S_n^L(t;\alpha f_1+\beta f_2,g)
        = \alpha S_n^L(t;f_1,g)+\beta S_n^L(t;f_2,g).
    \]
    Letting $n\to\infty$ shows that $I_{\alpha f_1+\beta f_2,g}^{\T}(t)$ exists and satisfies
    \[
        I_{\alpha f_1+\beta f_2,g}^{\T}(t)
        = \alpha I_{f_1,g}^{\T}(t)+\beta I_{f_2,g}^{\T}(t),
        \qquad t\in\T.
    \]
    The right-hand side has a continuous extension to $[0,1]$, namely $\alpha I_{f_1,g}+\beta I_{f_2,g}$. Hence, $I_{\alpha f_1+\beta f_2,g}$ exists on $[0,1]$ and the identity holds everywhere on $[0, 1]$ by density of $\T$.

    The proof of linearity in the second variable is identical, using the linearity of the finite dyadic left sums in the second variable.

    The bilinearity identities for interval integrals follow by subtracting the corresponding identities at $s$ and at $t$. Additivity follows directly from the definition of the interval integral:
    \[
        \int_r^s f\,dg+\int_s^t f\,dg = \big(I_{f,g}(s)-I_{f,g}(r)\big)+\big(I_{f,g}(t)-I_{f,g}(s)\big) = I_{f,g}(t)-I_{f,g}(r)= \int_r^t f\,dg.
    \]
    The identities $\int_s^s f\,dg=0$ and $\int_s^t f\,dg=-\int_t^s f\,dg$ are immediate.
\end{proof}

\begin{theorem}[Integration by parts]   \label{thm:extended-integration-by-parts}
    Let $p,q>1$ be conjugates, and let $f,g\in\Cn$. Suppose that $\zetaen{f}{p}\in\ell^p(\N_{-1})$ and $\zetaen{g}{q}\in\ell^q(\N_{-1})$. If $I_{f,g}$ exists on $[0,1]$, then $I_{g,f}^{\T}$ also admits a unique continuous extension to $[0,1]$, given by
    \begin{equation}\label{eq:ibp-definition-of-reverse-integral}
        I_{g,f}(t):=f(t)g(t)-f(0)g(0)-I_{f,g}(t), \qquad t\in[0,1].
    \end{equation}
    Equivalently, for every $s,t\in[0,1]$, we have the integration by parts formula:
    \[
        \int_s^t f\,dg+\int_s^t g\,df=f(t)g(t)-f(s)g(s).
    \]
\end{theorem}

\begin{proof}
    By Corollary~\ref{cor:separated-coefficient-criterion}, both integral maps $I_{f,g}^{\T}$ and $I_{g,f}^{\T}$ are well-defined on $\T$.

    We first prove the integration-by-parts identity on $\T$. Fix $t\in\T$ and let $n\ge\ell(t)$. At level $n$, the finite dyadic sums satisfy
    \[
        S_n^L(t;f,g)+S_n^L(t;g,f)
        = f(t)g(t)-f(0)g(0) - \sum_{k=0}^{2^nt-1}\big(f(t_{n,k+1})-f(t_{n,k})\big) \big(g(t_{n,k+1})-g(t_{n,k})\big).
    \]
    It remains to show that the last term converges to zero. By H\"older's inequality,
    \begin{align*}
        &\bigg|\sum_{k=0}^{2^nt-1}\big(f(t_{n,k+1})-f(t_{n,k})\big)\big(g(t_{n,k+1})-g(t_{n,k})\big)\bigg|
        \\
        & \qquad \qquad \qquad \le \bigg(\sum_{k\in I_n}\big|f(t_{n,k+1})-f(t_{n,k})\big|^p\bigg)^{\frac1p}
        \bigg(\sum_{k\in I_n}\big|g(t_{n,k+1})-g(t_{n,k})\big|^q\bigg)^{\frac1q}.
    \end{align*}
    Since $f(t_{n,k+1})-f(t_{n,k})=f_{n+1,L}(k)+f_{n+1,R}(k)$, Lemma~\ref{lem:child-estimate} gives
    \[
        \bigg(\sum_{k\in I_n}\big|f(t_{n,k+1})-f(t_{n,k})\big|^p\bigg)^{\frac1p}
        \le \mathcal A^{(p)}_f(n),
    \]
    and a similar bound for $g$. Hence,
    \[
        \bigg|\sum_{k=0}^{2^nt-1}\big(f(t_{n,k+1})-f(t_{n,k})\big)\big(g(t_{n,k+1})-g(t_{n,k})\big)\bigg|
        \le \mathcal A^{(p)}_f(n) \, \mathcal A^{(q)}_g(n).
    \]
    As in the proof of Corollary~\ref{cor:separated-coefficient-criterion}, the conditions $\zetaen{f}{p}\in\ell^p(\N_{-1})$ and $\zetaen{g}{q}\in\ell^q(\N_{-1})$ imply 
    \[
        \mathcal A^{(p)}_f\,\mathcal A^{(q)}_g\in\ell^1(\N_0), \qquad \text{and in particular} \qquad \mathcal A^{(p)}_f(n)\mathcal A^{(q)}_g(n)\xrightarrow{n\to\infty}0.
    \]
    Letting $n\to\infty$ in the finite dyadic sums identity gives
    \[
        I_{f,g}^{\T}(t)+I_{g,f}^{\T}(t)=f(t)g(t)-f(0)g(0), \qquad t\in\T.
    \]

    Now define $I_{g,f}$ by \eqref{eq:ibp-definition-of-reverse-integral}. This function is continuous on $[0,1]$ and agrees with $I_{g,f}^{\T}$ on the dense set $\T$. Hence, it is the unique continuous extension of $I_{g,f}^{\T}$. The displayed integration-by-parts identities follow immediately.
\end{proof}

\begin{remark}[One-sided tail condition]
\label{rem:one-sided-tail-condition}
    The Schauder tail condition is one-sided: it is enough to impose it on either the integrand or the integrator to obtain both Faber--Schauder integrals. More precisely, let $p,q>1$ be conjugates, and let $f,g\in\Cn$ satisfy $\zetaen{f}{p}\in\ell^p(\N_{-1})$ and $\zetaen{g}{q}\in\ell^q(\N_{-1})$. If either $H_f\in\ell^p(\N_0)$ or $H_g\in\ell^q(\N_0)$, then both $I_{f,g}$ and $I_{g,f}$ exist on $[0,1]$, and the integration by parts formula \eqref{eq:ibp-definition-of-reverse-integral} holds. Indeed, if $H_f\in\ell^p(\N_0)$, then Theorem~\ref{thm:uniform-continuity} gives the existence of $I_{f,g}$, and Theorem~\ref{thm:extended-integration-by-parts} gives the existence of $I_{g,f}$. The case $H_g\in\ell^q(\N_0)$ follows by first applying Theorem~\ref{thm:uniform-continuity} to the pair $(g,f)$, and then applying Theorem~\ref{thm:extended-integration-by-parts} with the roles of $f$ and $g$ interchanged.
\end{remark}

\medskip

\subsection{A dyadic analogue of the Young--Lo\`eve estimate}   \label{subsec:dyadic-young-loeve}

We now state and prove a dyadic analogue of the Young--Lo\`eve estimate.

First, to simplify notations, for $r>1$ and $h\in\Cn$ satisfying $\zetaen{h}{r}\in\ell^r(\N_{-1})$ and $H_h\in\ell^r(\N_0)$, we set
\begin{equation}    \label{def:D-and-L}
    D^{(r)}_h(N):=\mathcal A^{(r)}_h(N)+\frac12H_h(N), \qquad L^{(r)}_h(N):=\Vert\mathcal A^{(r)}_h\Vert_{\ell^r_{\ge N}}+\frac12\Vert H_h\Vert_{\ell^r_{\ge N}}.
\end{equation}

Throughout this subsection, $p,q>1$ are conjugates, and $f,g\in\Cn$ satisfy
\[
    \zetaen{f}{p}\in\ell^p(\N_{-1}), \qquad \zetaen{g}{q}\in\ell^q(\N_{-1}), \qquad H_f\in\ell^p(\N_0), \qquad H_g\in\ell^q(\N_0).
\]
Thus, by Theorem~\ref{thm:uniform-continuity}, both Faber--Schauder integrals $I_{f,g}$ and $I_{g,f}$ exist on $[0,1]$. Moreover, $L^{(p)}_f(N)<\infty$ and $L^{(q)}_g(N)<\infty$ for every $N\ge0$, and
\[
    L^{(p)}_f(N)\xrightarrow{N\to\infty}0, \qquad
    L^{(q)}_g(N)\xrightarrow{N\to\infty}0.
\]

\begin{lemma}[Projection estimate for the Faber--Schauder integral] \label{lem:projection-estimate-integral}
    Let $h,g\in\Cn$. Suppose
    \[
        \zetaen{h}{p}\in\ell^p(\N_{-1}), \qquad
        \zetaen{g}{q}\in\ell^q(\N_{-1}), \qquad
        H_h\in\ell^p(\N_0),
    \]
    such that $I_{h,g}$ exists on $[0,1]$. For $N\ge0$ and $u\in[0,1]$, set $u_N:=2^{-N}\lfloor2^Nu\rfloor$. Then
    \begin{align*}
        \Big|I_{h,g}(u)-I_{h,g}(u_N)-h(u)\big(g(u)-g(u_N)\big)\Big|
        \le \frac12\Vert\mathcal A^{(q)}_g\Vert_{\ell^q_{\ge N}} L^{(p)}_h(N) + \frac14 \big\Vert \mathcal A^{(p)}_h \, \mathcal A^{(q)}_g \big\Vert_{\ell^1_{\ge N}}.
    \end{align*}
\end{lemma}

\begin{proof}
    If $u\in\T$, the estimate is exactly \eqref{eq:projection-bound-on-T} applied to the pair $(h,g)$.

    For a general $u\in[0,1)$, consider $u_m=2^{-m}\lfloor2^m u\rfloor$ for $m\ge N$. Then $u_m\in\T$, $u_m\to u$, and
    \[
        2^{-N}\lfloor2^Nu_m\rfloor=u_N.
    \]
    Applying the estimate already proved for dyadic points to $u_m$ gives
    \begin{align*}
        \Big|I_{h,g}(u_m)-I_{h,g}(u_N)-h(u_m)\big(g(u_m)-g(u_N)\big)\Big|
        \le \frac12\Vert\mathcal A^{(q)}_g \Vert_{\ell^q_{\ge N}} L^{(p)}_h(N) + \frac14 \big\Vert \mathcal A^{(p)}_h \, \mathcal A^{(q)}_g \big\Vert_{\ell^1_{\ge N}}.
    \end{align*}
    Letting $m\to\infty$ and using the continuity of $I_{h,g}$, $h$, and $g$ gives the estimate at $u$. The case $u=1$ is trivial because $u_N=1$ and the left-hand side is zero.
\end{proof}

\begin{theorem}[A dyadic analogue of the Young--Lo\`eve estimate] \label{thm:dyadic-young-loeve}
    Let $p,q>1$ be conjugates, and let $f,g\in\Cn$. Suppose that $\zetaen{f}{p}\in\ell^p(\N_{-1})$, $\zetaen{g}{q}\in\ell^q(\N_{-1})$, $H_f\in\ell^p(\N_0)$, and $H_g\in\ell^q(\N_0)$. Then the
    Faber--Schauder integral $I_{f,g}$ exists on $[0,1]$, and for every $0\le s\le\xi\le t\le1$ and every $N\ge0$ satisfying $t-s\le2^{-N}$,
    \[
        \bigg|\int_s^t f\,dg-f(\xi)\big(g(t)-g(s)\big)\bigg| \le 8L^{(p)}_f(N)L^{(q)}_g(N).
    \]
    In particular, taking $N=0$ gives the global estimate
    \begin{equation}    \label{ineq:global-Young-Loeve}
        \bigg|\int_s^t f\,dg-f(\xi)\big(g(t)-g(s)\big)\bigg| \le 8L^{(p)}_f(0)L^{(q)}_g(0), \qquad 0\le s\le\xi\le t\le1.
    \end{equation}
\end{theorem}

We refer to these estimates as dyadic Young--Lo\`eve estimates, by analogy with the classical Young--Lo\`eve estimate for finite $p$-variation paths. See Young~\cite[Section~10]{Young1936} for the classical inequality; for a modern $p$-variation formulation, see \cite[Theorem~6.8]{FrizVictoir2010}.

The estimates in Theorem~\ref{thm:dyadic-young-loeve} come in both global and local forms. The right-hand side of the global estimate \eqref{ineq:global-Young-Loeve} is independent of the interval $[s,t]$ and of the evaluation point $\xi$. By contrast, the local estimate is scale-sensitive: since $L^{(p)}_f(N)$ and $L^{(q)}_g(N)$ decay to zero as $N\to\infty$, it yields
\[
    \sup_{\substack{0\le s\le\xi\le t\le1\\ t-s\le2^{-N}}} \left| \int_s^t f\,dg-f(\xi)\big(g(t)-g(s)\big) \right| \xrightarrow{N\to\infty}0.
\]

\begin{proof}
    Set $h:=f-f(\xi)$. Then
    \[
        \zetaen{h}{p}=\zetaen{f}{p}, \qquad \mathcal A^{(p)}_h=\mathcal A^{(p)}_f, \qquad H_h=H_f.
    \]
    By the bilinearity of Faber--Schauder integrals,
    \[
        \int_s^t f\,dg-f(\xi)\big(g(t)-g(s)\big) = \int_s^t h\,dg.
    \]
    Consider $s_N:=2^{-N}\lfloor2^Ns\rfloor$ and $t_N:=2^{-N}\lfloor2^Nt\rfloor$. Since $s\le t$ and $t-s\le2^{-N}$, the dyadic points $s_N$ and $t_N$ are either equal or adjacent level-$N$ dyadic points. We decompose
    \[
        \int_s^t h\,dg = \big(I_{h,g}(t)-I_{h,g}(t_N)\big) + \big(I_{h,g}(t_N)-I_{h,g}(s_N)\big) + \big(I_{h,g}(s_N)-I_{h,g}(s)\big).
    \]

    We first estimate the two boundary terms. By Lemma~\ref{lem:projection-estimate-integral}, applied at $t$ and at $s$,
    \[
        \Big|I_{h,g}(t)-I_{h,g}(t_N)-h(t)\big(g(t)-g(t_N)\big)\Big| + \Big|I_{h,g}(s)-I_{h,g}(s_N)-h(s)\big(g(s)-g(s_N)\big)\Big|
    \]
    is bounded by
    \[
        \Vert\mathcal A^{(q)}_g\Vert_{\ell^q_{\ge N}}L^{(p)}_f(N) + \frac12 \big\Vert \mathcal A^{(p)}_f \, \mathcal A^{(q)}_g \big\Vert_{\ell^1_{\ge N}}.
    \]
    Since $s\le\xi\le t$ and $t-s\le2^{-N}$, the oscillation estimate \eqref{ineq:oscillation-estimate} in Lemma~\ref{lem:dyadic-approach-f}, applied with $r=p$ to $h$, gives
    \[
        |h(s)| = |f(s) - f(\xi)| \le2D^{(p)}_f(N), \qquad |h(t)| = |f(t) - f(\xi)| \le2D^{(p)}_f(N).
    \]
    The endpoint estimate \eqref{ineq:two-endpoint-estimates} in Lemma~\ref{lem:dyadic-approach-f}, applied with $r=q$ to $g$, gives
    \[
        \big|g(s)-g(s_N)\big|\le D^{(q)}_g(N), \qquad
        \big|g(t)-g(t_N)\big|\le D^{(q)}_g(N).
    \]
    Hence, the product part of the two boundary terms is bounded by
    \[
        |h(t)|\big|g(t)-g(t_N)\big| + |h(s)|\big|g(s)-g(s_N)\big|
        \le 4D^{(p)}_f(N)D^{(q)}_g(N).
    \]

    It remains to estimate the middle term $I_{h,g}(t_N)-I_{h,g}(s_N)$. If $s_N=t_N$, then it is zero. Otherwise we can assume that $s_N$ and $t_N$ are adjacent level-$N$ dyadic points. Lemma~\ref{lem:refinement-remainders} (i) gives
    \[
        I_{h,g}^{\T}(t_N)-I_{h,g}^{\T}(s_N) = h(s_N)\big(g(t_N)-g(s_N)\big)+R_N,
    \]
    where
    \[
        |R_N| \le \frac14 \big\Vert \mathcal A^{(p)}_f \, \mathcal A^{(q)}_g \big\Vert_{\ell^1_{\ge N}}.
    \]
    We claim that
    \[
        |h(s_N)|=\big|f(s_N)-f(\xi)\big|\le2D^{(p)}_f(N).
    \]
    Indeed, since $s_N$ and $t_N$ are adjacent, if $\xi\le t_N$, then $\xi$ lies in the level-$N$ interval with left endpoint $s_N$, so the endpoint estimate \eqref{ineq:two-endpoint-estimates} in Lemma~\ref{lem:dyadic-approach-f} gives
    \[
        \big|f(s_N)-f(\xi)\big|\le D^{(p)}_f(N).
    \]
    If $\xi\ge t_N$, then
    \[
        \big|f(s_N)-f(\xi)\big| \le \big|f(s_N)-f(t_N)\big|+\big|f(t_N)-f(\xi)\big| \le 2D^{(p)}_f(N).
    \]
    Thus, the claim holds. Also, $g(t_N)-g(s_N)$ is the increment of $g$ over one level-$N$ dyadic interval, so Lemma~\ref{lem:child-estimate} gives
    \[
        \big|g(t_N)-g(s_N)\big|\le \mathcal A^{(q)}_g(N).
    \]
    Therefore, the middle term is bounded by
    \[
        2D^{(p)}_f(N)\mathcal A^{(q)}_g(N) + \frac14 \big\Vert \mathcal A^{(p)}_f \, \mathcal A^{(q)}_g \big\Vert_{\ell^1_{\ge N}}.
    \]

    Combining the boundary and middle estimates yields
    \[
        \bigg|\int_s^t h\,dg\bigg| \le 4D^{(p)}_f(N)D^{(q)}_g(N) + 2D^{(p)}_f(N)\mathcal A^{(q)}_g(N) + \Vert\mathcal A^{(q)}_g\Vert_{\ell^q_{\ge N}}L^{(p)}_f(N) + \frac34 \big\Vert \mathcal A^{(p)}_f \, \mathcal A^{(q)}_g \big\Vert_{\ell^1_{\ge N}}.
    \]
    Note that
    \[
        D^{(p)}_f(N) \le L^{(p)}_f(N), \qquad \mathcal A^{(p)}_f(N) \le \Vert\mathcal A^{(p)}_f\Vert_{\ell^p_{\ge N}} \le L^{(p)}_f(N),
    \]
    with similar inequalities for $f, p$ replaced by $g, q$. By H\"older's inequality,
    \[
        \big\Vert \mathcal A^{(p)}_f \, \mathcal A^{(q)}_g \big\Vert_{\ell^1_{\ge N}}
        \le \Vert\mathcal A^{(p)}_f\Vert_{\ell^p_{\ge N}} \Vert\mathcal A^{(q)}_g\Vert_{\ell^q_{\ge N}}
        \le L^{(p)}_f(N)L^{(q)}_g(N),
    \]
    thus, we obtain
    \[
        \bigg|\int_s^t h\,dg\bigg| \le \Big(4+2+1+\frac34\Big)L^{(p)}_f(N)L^{(q)}_g(N) \le 8L^{(p)}_f(N)L^{(q)}_g(N).
    \]
\end{proof}

\medskip

\subsection{Energy spaces}  \label{subsec:energy-spaces}

The previous subsections show that two types of coefficient conditions are needed for the full Faber--Schauder integral theory. The level energy condition $\zetaen{f}{p}\in\ell^p(\N_{-1})$ controls the dyadic endpoint integrals, while the Schauder tail condition $H_f\in\ell^p(\N_0)$ ensures continuous extension to the whole interval (Definitions~\ref{def:level-energies}, \ref{def:schauder-tail-sequence}). We now package these two conditions into a single space.

\begin{definition}[Energy spaces]   \label{def:energy-spaces}
    Let $p>1$. For $f\in\Cn$, define
    \[
        \|f\|_{\EE^p}:=\|f\|_{\infty}+\Vert\zetaen{f}{p}\Vert_{\ell^p(\N_{-1})}+\Vert H_f\Vert_{\ell^p(\N_0)},
    \]
    and
    \[
        \EE^p:=\{f\in\Cn:\|f\|_{\EE^p}<\infty\}.
    \]
\end{definition}

\begin{proposition}[Completeness of the energy space]   \label{prop:energy-banach}
    For every $p>1$, the space $(\EE^p,\|\cdot\|_{\EE^p})$ is a Banach space.
\end{proposition}

\begin{proof}
    It is straightforward to check, using the linearity of the Faber--Schauder coefficients and the subadditivity of the level energies and Schauder tails, that $\EE^p$ is a vector space and that $\|\cdot\|_{\EE^p}$ is a norm.
    
    Let $(f_m)_{m\ge1}$ be a Cauchy sequence in $\EE^p$. Then $(f_m)$ is Cauchy in the uniform norm, and hence converges uniformly to some $f\in\Cn$. For each fixed coefficient index $(n,k)$, the coefficient formulae \eqref{eq:FS-coeff-1}, \eqref{eq:FS-coeff} show that the map $h\mapsto\theta_{n,k}^h$ is continuous with respect to the uniform norm. Hence
    \[
        \theta_{n,k}^{f_m-f} \longrightarrow 0 \qquad \text{as } m\to\infty
    \]
    for every fixed $(n,k)$.

    We first prove the convergence of the level energy part. Fix $m$ and $N\ge0$. Since $f_\ell\to f$ uniformly,
    \[
        \sum_{n=-1}^{N}\big(\zetaen{f_m-f}{p}(n)\big)^p
        = \lim_{\ell\to\infty} \sum_{n=-1}^{N} \big(\zetaen{f_m-f_\ell}{p}(n)\big)^p.
    \]
    Therefore
    \[
        \sum_{n=-1}^{N}\big(\zetaen{f_m-f}{p}(n)\big)^p
        \le \liminf_{\ell\to\infty}\Vert\zetaen{f_m-f_\ell}{p}\Vert_{\ell^p(\N_{-1})}^p.
    \]
    Since $(f_m)$ is Cauchy in $\EE^p$, the right-hand side tends to $0$ as $m\to\infty$, uniformly in $N$. Letting $N\to\infty$ gives
    \[
        \Vert\zetaen{f_m-f}{p}\Vert_{\ell^p(\N_{-1})}\longrightarrow0.
    \]

    We next prove the convergence of the Schauder tail part. For fixed $m$, $n\ge0$, and $M\ge n$,
    \[
        \sum_{r=n}^{M}2^{-\frac r2}\sup_{k\in I_r}\big|\theta_{r,k}^{f_m-f}\big|
        = \lim_{\ell\to\infty} \sum_{r=n}^{M}2^{-\frac r2}\sup_{k\in I_r}\big|\theta_{r,k}^{f_m-f_\ell}\big|,
    \]
    hence
    \[
        \sum_{r=n}^{M}2^{-\frac r2}\sup_{k\in I_r}\big|\theta_{r,k}^{f_m-f}\big|
        \le
        \liminf_{\ell\to\infty}H_{f_m-f_\ell}(n).
    \]
    Letting $M\to\infty$ yields
    \[
        H_{f_m-f}(n)\le\liminf_{\ell\to\infty}H_{f_m-f_\ell}(n), \qquad n\ge0.
    \]
    Therefore, for every $N\ge0$,
    \[
        \sum_{n=0}^{N} \big(H_{f_m-f}(n)\big)^p
        \le \liminf_{\ell\to\infty}\sum_{n=0}^{N} \big(H_{f_m-f_\ell}(n)\big)^p
        \le \liminf_{\ell\to\infty}\Vert H_{f_m-f_\ell}\Vert_{\ell^p(\N_0)}^p.
    \]
    Since $(f_m)$ is Cauchy in $\EE^p$, the right-hand side tends to $0$ as $m\to\infty$, uniformly in $N$. Letting $N\to\infty$ gives
    \[
        \Vert H_{f_m-f}\Vert_{\ell^p(\N_0)}\longrightarrow0.
    \]

    Combining uniform convergence, level energy convergence, and Schauder tail convergence, we obtain
    \[
        \|f_m-f\|_{\EE^p}\longrightarrow0.
    \]
    In particular, $f\in\EE^p$, and $\EE^p$ is complete.
\end{proof}

The next proposition compares the relative size of the energy spaces $\EE^p$ and shows that they are nested with respect to the exponent $p$.

\begin{proposition}[Monotonicity of the energy spaces]  \label{prop:energy-space-monotonicity}
    Let $1<p<q<\infty$. Then $\EE^p\subset \EE^q$ continuously. More precisely, there exists a constant $C_{p,q}>0$ such that
    \[
        \|f\|_{\EE^q}\le C_{p,q}\|f\|_{\EE^p}, \qquad f\in\EE^p.
    \]
    Furthermore, the inclusion is strict.
\end{proposition}

\begin{proof}
    For $n\ge0$, the monotonicity of $\ell^r$ norms gives
    \[
        \zetaen{f}{q}(n) = 2^{-\frac n2}\bigg(\sum_{k\in I_n}\big|\theta^f_{n,k}\big|^q\bigg)^{\frac1q}
        \le 2^{-\frac n2}\bigg(\sum_{k\in I_n}\big|\theta^f_{n,k}\big|^p\bigg)^{\frac1p}
        = \zetaen{f}{p}(n).
    \]
    The affine level is controlled by a constant depending only on $p$ and $q$. Hence
    \[
        \|\zetaen{f}{q}\|_{\ell^q(\N_{-1})} \le C_{p,q}\|\zetaen{f}{p}\|_{\ell^p(\N_{-1})}
    \]
    holds for some constant $C_{p,q} > 0$ depending only on $p, q$. Also, since $p<q$,
    \[
        \|H_f\|_{\ell^q(\N_0)} \le \|H_f\|_{\ell^p(\N_0)}.
    \]
    Together with the uniform norm term, this proves the continuous embedding.

    To see that the inclusion is strict, set
    \[
        a_n:=2^{-\frac nq}(n+1)^{-\frac2q}, \qquad n\ge0,
    \]
    and define
    \[
        f(t):=\sum_{n=0}^{\infty}\sum_{k\in I_n}2^{\frac n2}a_ne_{n,k}(t), \qquad t\in[0,1].
    \]
    The series converges uniformly. Indeed, since
    \[
        \Vert e_{n,k}\Vert_{\infty}\le \frac12 2^{-\frac n2},
    \]
    and since at each point $t\in[0,1]$ at most one Faber--Schauder function at level $n$ is nonzero, the level-$n$ contribution is bounded in the uniform norm by $\frac12 a_n$. Since $\sum_{n=0}^{\infty}a_n<\infty$, the series converges uniformly. Moreover,
    \[
        \big(\zetaen{f}{q}(n)\big)^q
        = 2^{-\frac{nq}{2}}\sum_{k\in I_n}\big|2^{\frac n2}a_n\big|^q = 2^na_n^q = (n+1)^{-2},
    \]
    so $\zetaen{f}{q}\in\ell^q(\N_0)$. Also,
    \[
        H_f(n)=\sum_{m=n}^{\infty}a_m
        \le C_q2^{-\frac nq}(n+1)^{-\frac2q},
    \]
    and hence $H_f\in\ell^q(\N_0)$. Thus $f\in\EE^q$.

    On the other hand,
    \[
        \big(\zetaen{f}{p}(n)\big)^p
        = 2^{-\frac{np}{2}}\sum_{k\in I_n}\big|2^{\frac n2}a_n\big|^p
        = 2^na_n^p = 2^{n(1-\frac pq)}(n+1)^{-\frac{2p}{q}},
    \]
    which tends to infinity as $n\to\infty$. Therefore $\zetaen{f}{p}\notin\ell^p(\N_0)$, and hence $f\notin\EE^p$. This proves that the inclusion is strict.
\end{proof}

The following result summarizes that membership in the energy space for both the integrand and the integrator is sufficient for the full Faber--Schauder integral theory.

\begin{corollary}[Energy-space criterion] \label{cor:energy-space-full-theory}
    Let $p,q>1$ satisfy
    \[
        \frac1p+\frac1q\ge1.
    \]
    If $f\in\EE^p$ and $g\in\EE^q$, then both Faber--Schauder integrals $I_{f,g}$ and $I_{g,f}$ exist on $[0,1]$.

    More precisely, there exist conjugate exponents $\bar p,\bar q>1$ such that
    \[
        \bar p\ge p, \qquad \bar q\ge q, \qquad f\in\EE^{\bar p}, \qquad g\in\EE^{\bar q}.
    \]
    With respect to such a conjugate pair $(\bar p,\bar q)$, the algebraic properties in Proposition~\ref{prop:extended-bilinearity-additivity}, the integration-by-parts identity in Theorem~\ref{thm:extended-integration-by-parts}, and the dyadic Young--Lo\`eve estimate in Theorem~\ref{thm:dyadic-young-loeve} all apply to the pair $(f,g)$.
\end{corollary}

\begin{proof}
    By Proposition~\ref{prop:energy-space-monotonicity}, we can choose conjugate exponents $\bar p,\bar q>1$ with $\bar p \ge p$, $\bar q \ge q$ such that $f\in\EE^p\subset\EE^{\bar p}$ and $g\in\EE^q\subset\EE^{\bar q}$. Therefore, Theorem~\ref{thm:uniform-continuity} gives the existence of $I_{f,g}$ on $[0,1]$. Applying the same theorem to the pair $(g,f)$ with the conjugate exponents $(\bar q,\bar p)$ gives the existence of $I_{g,f}$ on $[0,1]$. The remaining assertions follow directly from Proposition~\ref{prop:extended-bilinearity-additivity}, Theorem~\ref{thm:extended-integration-by-parts}, and Theorem~\ref{thm:dyadic-young-loeve}, applied with the conjugate pair $(\bar p,\bar q)$.
\end{proof}

Using the energy spaces of integrand and integrator, the next result shows that the Faber--Schauder integral map is a bounded bilinear operator.

\begin{proposition}[Bounded bilinear integration operator]  \label{prop:bounded-bilinear-integration-operator}
    Let $p,q>1$ satisfy $\frac1p+\frac1q\ge1$. The map
    \[
        \mathcal I_{p,q}:\EE^p\times\EE^q\to C([0,1]), \qquad \mathcal I_{p,q}(f,g):=I_{f,g},
    \]
    is a bounded bilinear operator. More precisely, there exists a constant $C_{p,q}>0$ such that
    \[
        \|I_{f,g}\|_{\infty}\le C_{p,q}\|f\|_{\EE^p}\|g\|_{\EE^q}, \qquad f\in\EE^p,\ g\in\EE^q.
    \]
\end{proposition}

\begin{proof}
    As in Corollary~\ref{cor:energy-space-full-theory}, choose conjugates $\bar p,\bar q>1$ such that $\bar p\ge p$, $\bar q\ge q$, $f\in\EE^p\subset\EE^{\bar p}$, and $g\in\EE^q\subset\EE^{\bar q}$. Thus, the map is well-defined, and its bilinearity follows from Proposition~\ref{prop:extended-bilinearity-additivity}. It remains to prove boundedness. By Theorem~\ref{thm:dyadic-young-loeve} with $s=\xi=0$, we have for every $t\in[0,1]$,
    \[
        \big|I_{f,g}(t)-f(0)\big(g(t)-g(0)\big)\big|\le 8L_f^{(p)}(0)L_g^{(q)}(0).
    \]
    Hence
    \[
        |I_{f,g}(t)|\le 2\|f\|_{\infty}\|g\|_{\infty}+8L_f^{(p)}(0)L_g^{(q)}(0).
    \]
    From the $\ell^p$-bound of $\mathcal A^{(p)}_f$ in \eqref{ineq:Ap-bound}, we have
    \[
        L_f^{(p)}(0)=\|\mathcal A^{(p)}_f\|_{\ell^p(\N_0)}+\frac12\|H_f\|_{\ell^p(\N_0)}\le C_p\|f\|_{\EE^p}, \qquad \text{with} \qquad C_p:=\|K_p\|_{\ell^1(\N_0)}+\frac12.
    \]
    Similarly,
    \[
        L_g^{(q)}(0)\le C_q\|g\|_{\EE^q}, \qquad C_q:=\|K_q\|_{\ell^1(\N_0)}+\frac12.
    \]
    Taking the supremum over $t\in[0,1]$ gives
    \[
        \|I_{f,g}\|_{\infty}\le \big(2+8C_pC_q\big)\|f\|_{\EE^p}\|g\|_{\EE^q}.
    \]
\end{proof}

The next result shows that any element of the space $\EE^p$ in Definition~\ref{def:energy-spaces} can be approximated by piecewise linear functions in the $\Vert \cdot \Vert_{\EE^p}$-norm.

\begin{proposition}[Density of finite Faber--Schauder expansions]   \label{prop:finite-fs-density}
    Let $p>1$. For $f\in\EE^p$ and $N\ge0$, define the level-$N$ Faber--Schauder truncation by
    \begin{equation}    \label{FS-truncation}
        \Pi_N f(t):=f(0)+\sum_{n=-1}^{N-1}\sum_{k\in I_n}\theta_{n,k}^f e_{n,k}(t), \qquad t\in[0,1].
    \end{equation}
    Then $\|f-\Pi_N f\|_{\EE^p} \xrightarrow{N\to\infty} 0$. In particular, finite Faber--Schauder expansions are dense in $\EE^p$.
\end{proposition}

\begin{proof}
    Set $h_N:=f-\Pi_N f$. Then the Faber--Schauder coefficients of $h_N$ vanish at levels $-1,0,\dots,N-1$, and agree with those of $f$ at levels $n\ge N$.

    First, by the tail estimate \eqref{ineq:N-dyadic-interpol}, for every $t\in[0,1]$,
    \[
        |h_N(t)|\le \frac12\sum_{m=N}^{\infty}2^{-\frac m2}\sup_{k\in I_m}\big|\theta_{m,k}^f\big|=\frac12H_f(N).
    \]
    Since $H_f\in\ell^p(\N_0)$, we have $H_f(N)\to0$, and hence
    \[
        \|h_N\|_{\infty}\xrightarrow{N\to\infty}0.
    \]

    The level energy part satisfies
    \[
        \|\zetaen{h_N}{p}\|_{\ell^p(\N_{-1})}^p=\sum_{n=N}^{\infty}\big(\zetaen{f}{p}(n)\big)^p\xrightarrow{N\to\infty}0.
    \]

    It remains to control the Schauder tail part. Since $h_N$ has no coefficients below level $N$,
    \[
        H_{h_N}(n)=H_f(\max\{n,N\}), \qquad n\ge0.
    \]
    Therefore
    \[
        \|H_{h_N}\|_{\ell^p(\N_0)}^p = \sum_{n=0}^{N-1} \big(H_f(N)\big)^p + \sum_{n=N}^{\infty} \big(H_f(n)\big)^p = N \big(H_f(N)\big)^p + \sum_{n=N}^{\infty} \big(H_f(n)\big)^p.
    \]
    The second term tends to zero because $H_f\in\ell^p(\N_0)$. For the first term, note that $H_f$ is nonincreasing. Hence
    \[
        \frac N2 \big(H_f(N)\big)^p\le \sum_{n=\lfloor N/2\rfloor}^{N-1} \big(H_f(n)\big)^p \xrightarrow{N\to\infty} 0.
    \]
    Thus $\|H_{h_N}\|_{\ell^p(\N_0)}\xrightarrow{N\to\infty}0$.
    
    Combining the three estimates gives $\|f - \Pi_N f\|_{\EE^p}\to0$ as $N \to \infty$.
\end{proof}

With the energy spaces in place, we now provide an analogue of Ciesielski's isomorphism theorem~\cite{Ciesielski:isomorphism}. In 1960, Ciesielski proved that the H\"older space $C^{\alpha}([0, 1])$ for $\alpha \in (0, 1)$ is isomorphic to a space $\ell^{\infty}$ of real sequences, equipped with the supremum norm:
\begin{align*}
    \big(C^{\alpha}([0, 1]), \Vert \cdot \Vert_{C^{\alpha}}\big) &\longrightarrow \big(\ell^{\infty}, \Vert \cdot \Vert_{\ell^\infty}\big)
    \\
    f \quad &\mapsto \Big(f(0), f(1)-f(0), \big(2^{n(\alpha-\frac12)} \theta^f_{n,k}\big)_{n\ge0,\ k\in I_n}\Big).
\end{align*}
In what follows, we prove an analogous coefficient-space representation for the Banach space $(\EE^p,\|\cdot\|_{\EE^p})$.

\begin{definition}[Coefficient energy space]
\label{def:coefficient-energy-space}
    For a coefficient family $a=(a_{n,k})_{n\in\N_{-1},\,k\in I_n}$, define
    \[
        \zeta_a^{(p)}(-1):=2^{1-\frac1p}|a_{-1,0}| \qquad \text{and} \qquad \zeta_a^{(p)}(n):=\Big(2^{-\frac{np}{2}}\sum_{k\in I_n}|a_{n,k}|^p\Big)^{\frac1p} \quad \text{for } n \ge 0.
    \]
    Also define
    \[
        H_a(n):=\sum_{m=n}^{\infty}2^{-\frac m2}\sup_{k\in I_m}|a_{m,k}|, \qquad n\ge0.
    \]
    We define the coefficient energy space by
    \[
        \mathfrak e^p:=\Big\{(c,a):c\in\R,\ \zeta_a^{(p)}\in\ell^p(\N_{-1}),\ H_a\in\ell^p(\N_0)\Big\},
    \]
    equipped with the norm
    \[
        \|(c,a)\|_{\mathfrak e^p}:=|c|+\|\zeta_a^{(p)}\|_{\ell^p(\N_{-1})}+\|H_a\|_{\ell^p(\N_0)}.
    \]
\end{definition}

\begin{theorem}[Coefficient-space representation]
\label{thm:coefficient-space-isomorphism}
    Let $p>1$. The Faber--Schauder coefficient map
    \[
        \mathscr C_p:\EE^p\to\mathfrak e^p, \qquad \mathscr C_p f:=\Big(f(0),\big(\theta_{n,k}^f\big)_{n\in\N_{-1},\,k\in I_n}\Big),
    \]
    is a Banach space isomorphism. Its inverse is the reconstruction map
    \[
        \mathscr R_p(c,a)(t):=c+\sum_{n=-1}^{\infty}\sum_{k\in I_n}a_{n,k}e_{n,k}(t), \qquad t\in[0,1].
    \]
\end{theorem}

\begin{proof}
    For $f\in\EE^p$, we have
    \[
        \|\mathscr C_p f\|_{\mathfrak e^p}=|f(0)|+\|\zetaen{f}{p}\|_{\ell^p(\N_{-1})}+\|H_f\|_{\ell^p(\N_0)}\le \|f\|_{\EE^p}.
    \]
    Hence $\mathscr C_p$ is bounded.

    Conversely, let $(c,a)\in\mathfrak e^p$. Since $H_a(0)<\infty$, the series
    \[
        \sum_{n=0}^{\infty}\sum_{k\in I_n}a_{n,k}e_{n,k}(t), \qquad t \in [0, 1]
    \]
    converges uniformly on $[0,1]$, because at each level $n\ge0$ at most one Faber--Schauder function $e_{n,k}$ is nonzero at a given point and its height is at most $2^{-\frac n2-1}$. Therefore $\mathscr R_p(c,a)$ is a well-defined continuous function.

    Moreover,
    \[
        \|\mathscr R_p(c,a)\|_{\infty}\le |c|+|a_{-1,0}|+\frac12H_a(0)\le |c|+\|\zeta_a^{(p)}\|_{\ell^p(\N_{-1})}+\frac12\|H_a\|_{\ell^p(\N_0)}.
    \]
    The Faber--Schauder coefficients of $\mathscr R_p(c,a)$ are precisely $a_{n,k}$, and its initial value is $c$. Hence
    \[
        \|\mathscr R_p(c,a)\|_{\EE^p}\le |c|+2\|\zeta_a^{(p)}\|_{\ell^p(\N_{-1})}+\frac32\|H_a\|_{\ell^p(\N_0)}\le 2\|(c,a)\|_{\mathfrak e^p}.
    \]
    Thus $\mathscr R_p$ maps $\mathfrak e^p$ boundedly into $\EE^p$.

    Finally, the identities
    \[
        \mathscr R_p\mathscr C_p f=f \quad \text{for} \quad f\in\EE^p, \qquad \text{and} \qquad 
        \mathscr C_p\mathscr R_p(c,a)=(c,a) \quad \text{for} \quad (c,a)\in\mathfrak e^p,
    \]
    follow from the uniqueness of the Faber--Schauder expansion. Therefore, $\mathscr C_p$ is a bounded linear bijection with bounded inverse $\mathscr R_p$. Since $\EE^p$ is Banach by Proposition~\ref{prop:energy-banach}, $\mathfrak e^p$ is Banach as well, and $\mathscr C_p$ is a Banach space isomorphism.
\end{proof}

We conclude this subsection with a remark on the local nature of the energy condition. Although the energy condition is defined through the Faber--Schauder expansion on $[0, 1]$, it is local with respect to finite dyadic decompositions. In particular, the construction of the Faber--Schauder integral is stable under extension to dyadically compatible larger intervals.

\begin{remark}[Dyadic locality of the energy condition]
\label{rem:dyadic-locality}
    Although the Faber--Schauder expansion is defined relative to a fixed interval, the energy conditions defining $\EE^p$ are local in the dyadic sense. For example, suppose that a function $F$ is defined on $[0,2]$, and consider the Haar--Schauder system on $[0,2]$. Then, from level $1$ onward, each Schauder function is supported either in $[0,1]$ or in $[1,2]$. The only coefficients which interact with the boundary point $1$ are the affine coefficient and the finitely many coarse coefficients at the first levels.

    Consequently, the conditions $\zetaen{F}{p}\in\ell^p$ and $H_F\in\ell^p$ are unaffected by these finitely many coarse coefficients. If the restriction of $F$ to $[0,1]$ is already known to satisfy the energy conditions, then verifying the energy conditions on the enlarged interval reduces to checking the tail coefficients supported on the newly added dyadic piece $[1,2]$. These coefficients are determined by the same second-order increment formula as in \eqref{eq:FS-coeff}, applied locally on $[1,2]$.

    Thus, if functions $f,g$ on $[0,1]$ are extended to functions $F,G$ on a dyadically compatible larger interval, one does not need to recompute all coarse coefficients of the enlarged expansion in order to verify the energy conditions. It is enough to check the tail coefficients on the newly added dyadic pieces; the finitely many coefficients crossing the old boundary do not affect membership in the energy spaces.

    The integral itself reflects the same locality. On the original interval $[0,1]$, the dyadic left sums computed in the enlarged dyadic system agree with the original dyadic left sums after passing to a cofinal subsequence of dyadic levels. Hence, on $[0,1]$, the Faber--Schauder integral constructed on the enlarged interval coincides with the original Faber--Schauder integral.
\end{remark}

\medskip

\subsection{Approximation by classical Riemann--Stieltjes integrals}
\label{subsec:rs-approximation}

The density of finite Faber--Schauder expansions (Proposition~\ref{prop:finite-fs-density}) and the bounded bilinearity of the Faber--Schauder integral (Proposition~\ref{prop:bounded-bilinear-integration-operator}) imply that the Faber--Schauder integral can be recovered as a uniform limit of classical Riemann--Stieltjes integrals. This gives a direct approximation interpretation of the construction.

\begin{theorem}[Approximation by Riemann--Stieltjes integrals]  \label{thm:classical-rs-approximation}
    Let $p,q>1$ be conjugates, and let $f\in\EE^p$, $g\in\EE^q$. For $N,M\ge0$, consider continuous piecewise linear functions $\Pi_N f$ and $\Pi_M g$ from \eqref{FS-truncation}. Then
    \[
        \sup_{t\in[0,1]}\left|\int_0^t \Pi_N f\,d(\Pi_M g)-I_{f,g}(t)\right|\xrightarrow{N,M\to\infty}0,
    \]
    where the integral on the left-hand side is the classical Riemann--Stieltjes integral.
\end{theorem}

\begin{proof}
    Since $\Pi_N f$ and $\Pi_M g$ are continuous piecewise linear functions on $[0,1]$, the function $\Pi_N f$ is continuous and $\Pi_M g$ has bounded variation. Hence, the classical Riemann--Stieltjes integral
    \[
        \int_0^t \Pi_N f\,d(\Pi_M g)
    \]
    exists for every $t\in[0,1]$.

    We first show that this classical integral agrees with the Faber--Schauder integral of the pair $(\Pi_N f,\Pi_M g)$. Fix $t\in\T$. For every $m\ge\ell(t)$, the sum $S_m^L(t;\Pi_N f,\Pi_M g)$ is a Riemann--Stieltjes sum over the dyadic partition of $[0,t]$, with mesh size $2^{-m}$. Since $\Pi_N f$ is continuous and $\Pi_M g$ has bounded variation, these sums converge to the classical Riemann--Stieltjes integral. Therefore
    \[
        I_{\Pi_N f,\Pi_M g}^{\T}(t)=\int_0^t \Pi_N f\,d(\Pi_M g), \qquad t\in\T.
    \]
    Both sides are continuous functions of $t$, so the identity extends to all $t\in[0,1]$:
    \[
        I_{\Pi_N f,\Pi_M g}(t)=\int_0^t \Pi_N f\,d(\Pi_M g), \qquad t\in[0,1].
    \]

    By Proposition~\ref{prop:finite-fs-density},
    \[
        \Pi_N f \xrightarrow{N\to\infty} f \quad\text{in }\EE^p, \qquad \Pi_M g \xrightarrow{M\to\infty} g \quad\text{in }\EE^q.
    \]
    Using bilinearity and Proposition~\ref{prop:bounded-bilinear-integration-operator}, we obtain
    \[
        \|I_{\Pi_N f,\Pi_M g}-I_{f,g}\|_{\infty}
        \le C_{p,q}\|\Pi_N f-f\|_{\EE^p}\|\Pi_M g\|_{\EE^q} + C_{p,q}\|f\|_{\EE^p}\|\Pi_M g-g\|_{\EE^q}.
    \]
    Since $\Pi_M g\to g$ in $\EE^q$, the sequence $\|\Pi_M g\|_{\EE^q}$ is bounded. Hence, the right-hand side converges to $0$ as $N,M\to\infty$. This proves the result.
\end{proof}

This approximation theorem is specific to the energy-space framework. In the classical Young theory, continuity of the integral is usually formulated with respect to H\"older or variation topologies, and approximation by smoother paths requires convergence in those topologies, which control not only the uniform distance between paths but also their regularity through H\"older seminorms or variation seminorms. Here, by contrast, the approximants are not auxiliary smooth paths chosen to converge in a strong topology; they are the canonical dyadic piecewise linear interpolants determined solely by finitely many observations of the paths on the dyadic grids. Theorem~\ref{thm:classical-rs-approximation} shows that the classical Riemann--Stieltjes integrals of these finite-dimensional approximations converge uniformly to the Faber--Schauder integral. Thus, the construction is compatible with classical bounded-variation integration at every finite level, while the limiting pair $(f,g)$ itself may lie outside the finite-variation Young regime.

\medskip

\subsection{Energy regularity of the Faber--Schauder integral}  \label{subsec:energy-regularity-integral}

We close this section with a permanence property of the Faber--Schauder integral. The general theory developed above gives
\[
    (f,g)\mapsto I_{f,g}
\]
as a bounded bilinear map from $\EE^p\times\EE^q$ into $C([0,1])$. Under an additional decay condition on the integrand, the integral inherits the energy regularity of the integrator.

For $r>1$ and $h\in\EE^r$, recall the quantity $L_h^{(r)}$ from \eqref{def:D-and-L}.

\begin{theorem}[Integrator-side energy regularity]
\label{thm:integral-energy-regularity}
    Let $p,q>1$ be conjugates, let $f\in\EE^p$ and $g\in\EE^q$. Suppose that there exist constants $C>0$ and $\rho>0$ such that
    \[
        L_f^{(p)}(N)\le C2^{-\rho N}, \qquad N\ge0.
    \]
    Then $I_{f,g}\in\EE^q$.
\end{theorem}

\begin{proof}
    Set $h(t):=I_{f,g}(t)$ for $t\in[0,1]$, and write
    \[
        P_{f,g}(N) := \sum_{n=N}^{\infty} \mathcal A_f^{(p)}(n) \, \mathcal A_g^{(q)}(n), \qquad N \ge 0.
    \]
    Since $\mathcal A_f^{(p)}(n)\le L_f^{(p)}(n)\le C2^{-\rho n}$ and $\mathcal A_g^{(q)}\in\ell^q(\N_0)$, H\"older's inequality gives for some $C_{p,q,\rho}>0$
    \[
        P_{f,g}(N)\le C\sum_{n=N}^{\infty}2^{-\rho n}\mathcal A_g^{(q)}(n)\le C_{p,q,\rho}\,2^{-\rho N}\Vert\mathcal A_g^{(q)}\Vert_{\ell^q_{\ge N}}.
    \]
    Since $N\mapsto\Vert\mathcal A_g^{(q)}\Vert_{\ell^q_{\ge N}}$ is nonincreasing, the same bound also gives
    \[
        \sum_{m=N}^{\infty}P_{f,g}(m) \le C_{p,q,\rho}2^{-\rho N}\Vert\mathcal A_g^{(q)}\Vert_{\ell^q_{\ge N}},
    \]
    possibly with a different constant $C_{p,q,\rho}$. In particular,
    \[
        P_{f,g}\in\ell^q(\N_0), \qquad \sum_{N=0}^{\infty}\bigg(\sum_{m=N}^{\infty}P_{f,g}(m)\bigg)^q<\infty.
    \]

    We first estimate the level energies of $h$. For a level-$n$ dyadic interval, consider three dyadic points $t_{n,k}$, $t_{n+1,2k+1}$, and $t_{n,k+1}$. By Lemma~\ref{lem:refinement-remainders}, applied to the two level-$(n+1)$ intervals $[t_{n,k},t_{n+1,2k+1}]$ and $[t_{n+1,2k+1},t_{n,k+1}]$, there are remainders $R_{n+1,2k}$ and $R_{n+1,2k+1}$ such that
    \[
        h(t_{n+1,2k+1})-h(t_{n,k}) = f(t_{n,k})\big(g(t_{n+1,2k+1})-g(t_{n,k})\big)+R_{n+1,2k},
    \]
    and
    \[
        h(t_{n,k+1})-h(t_{n+1,2k+1}) = f(t_{n+1,2k+1})\big(g(t_{n,k+1})-g(t_{n+1,2k+1})\big)+R_{n+1,2k+1}.
    \]
    Moreover, from the proof of Lemma~\ref{lem:refinement-remainders} we claim the level-wise estimate
    \begin{equation}    \label{level-wise-estimate}
        \bigg(\sum_{j\in I_N}\big|R_{N,j}\big|^q\bigg)^{\frac1q}\le \frac14 P_{f,g}(N), \qquad N\ge0.
    \end{equation}
    Indeed, for fixed $m\ge N$, the level-$m$ refinement contributions inside the level-$N$ intervals form disjoint blocks. If $\Delta_{m,j}$ denotes the contribution of level $m$ inside $[t_{N,j},t_{N,j+1}]$, then
    \[
        R_{N,j}=\sum_{m=N}^{\infty}\Delta_{m,j}, \qquad \Delta_{m,j} := \sum_{k:\,I_{m,k}\subset[t_{N,j},t_{N,j+1}]} f_{m+1,L}(k)g_{m+1,R}(k).
    \]
    By Minkowski's inequality,
    \[
        \bigg(\sum_{j\in I_N}|R_{N,j}|^q\bigg)^{\frac1q}
        \le \sum_{m=N}^{\infty}\bigg(\sum_{j\in I_N}|\Delta_{m,j}|^q\bigg)^{\frac1q}.
    \]
    For each fixed $m\ge N$, set $K_{m,j}:=\{k\in I_m:I_{m,k}\subset[t_{N,j},t_{N,j+1}]\}$ for $j\in I_N$. The sets $K_{m,j}$ are disjoint and their union is $I_m$. For $j\in I_N$, H\"older's inequality on the block $K_{m,j}$ gives
    \[
        |\Delta_{m,j}|
        \le \bigg(\sum_{k\in K_{m,j}}\big|f_{m+1,L}(k)\big|^p\bigg)^{\frac1p} \bigg(\sum_{k\in K_{m,j}}\big|g_{m+1,R}(k)\big|^q\bigg)^{\frac1q}.
    \]
    Write
    \[
        F_{m,j}:=\bigg(\sum_{k\in K_{m,j}}\big|f_{m+1,L}(k)\big|^p\bigg)^{\frac1p}, \qquad
        G_{m,j}:=\bigg(\sum_{k\in K_{m,j}}\big|g_{m+1,R}(k)\big|^q\bigg)^{\frac1q}.
    \]
    Then
    \[
        \bigg(\sum_{j\in I_N}|\Delta_{m,j}|^q\bigg)^{\frac1q}
        \le \bigg(\sum_{j\in I_N}F_{m,j}^qG_{m,j}^q\bigg)^{\frac1q}
        \le \Big( \sup_{j\in I_N}F_{m,j} \Big) \bigg(\sum_{j\in I_N}G_{m,j}^q\bigg)^{\frac1q}.
    \]
    Since
    \[
        \sup_{j\in I_N}F_{m,j}\le\bigg(\sum_{j\in I_N}F_{m,j}^p\bigg)^{\frac1p}
        = \bigg(\sum_{k\in I_m}\big|f_{m+1,L}(k)\big|^p\bigg)^{\frac1p},
    \]
    and
    \[
        \bigg(\sum_{j\in I_N}G_{m,j}^q\bigg)^{\frac1q}
        = \bigg(\sum_{k\in I_m}\big|g_{m+1,R}(k)\big|^q\bigg)^{\frac1q},
    \]
    Lemma~\ref{lem:child-estimate} yields
    \[
        \bigg(\sum_{j\in I_N}|\Delta_{m,j}|^q\bigg)^{\frac1q}
        \le \frac14\mathcal A_f^{(p)}(m) \, \mathcal A_g^{(q)}(m).
    \]
    Summing over $m\ge N$ gives the claimed estimate \eqref{level-wise-estimate}.

    Since
    \[
        2^{-\frac n2} \theta_{n,k}^h = \big(h(t_{n+1,2k+1})-h(t_{n,k})\big)-\big(h(t_{n,k+1})-h(t_{n+1,2k+1})\big),
    \]
    we obtain
    \begin{align*}
        2^{-\frac n2}\theta_{n,k}^h = f(t_{n,k})\,2^{-\frac n2}\theta_{n,k}^g &+ \big(f(t_{n,k})-f(t_{n+1,2k+1})\big) \big(g(t_{n,k+1})-g(t_{n+1,2k+1})\big) 
        \\ & + R_{n+1,2k}-R_{n+1,2k+1}.
    \end{align*}
    Taking the $\ell^q$ norm over $k\in I_n$ and using Lemma~\ref{lem:child-estimate} with the estimate \eqref{level-wise-estimate}, we get
    \[
        \zetaen{h}{q}(n)\le \Vert f\Vert_\infty\zetaen{g}{q}(n)+\frac14 \mathcal A_f^{(p)}(n) \, \mathcal A_g^{(q)}(n) +\frac12 P_{f,g}(n+1).
    \]
    Since $\zetaen{g}{q}\in\ell^q(\N_{-1})$, $\mathcal A_f^{(p)} \mathcal A_g^{(q)}\in\ell^q(\N_0)$, and $P_{f,g}\in\ell^q(\N_0)$, it follows that $\zetaen{h}{q}\in\ell^q(\N_{-1})$.

    It remains to estimate the Schauder tail of $h$. From the same coefficient identity and the pointwise remainder estimate in Lemma~\ref{lem:refinement-remainders}, we have
    \[
        2^{-\frac n2} \sup_{k\in I_n} \big|\theta_{n,k}^h\big|
        \le \Vert f\Vert_\infty\,2^{-\frac n2}\sup_{k\in I_n}\big|\theta_{n,k}^g\big|
        + \frac14 \mathcal A_f^{(p)}(n) \, \mathcal A_g^{(q)}(n) + \frac12 P_{f,g}(n+1).
    \]
    Therefore, for every $N\ge0$,
    \[
        H_h(N)\le \Vert f\Vert_\infty H_g(N)+\frac14 P_{f,g}(N)+\frac12\sum_{n=N}^{\infty}P_{f,g}(n+1).
    \]
    The first term belongs to $\ell^q(\N_0)$ because $g\in\EE^q$, and the last two terms belong to $\ell^q(\N_0)$ by the estimate on $P_{f,g}$ above. Hence, $H_h\in\ell^q(\N_0)$.
    Since $h$ is continuous and $h(1)-h(0)$ is finite, the affine level is finite. Therefore $h=I_{f,g}\in\EE^q$.
\end{proof}

The additional assumption in Theorem~\ref{thm:integral-energy-regularity} is not needed for the existence of the Faber--Schauder integral. Rather, it is used to ensure that the integral belongs to the same energy space as the integrator. This kind of regularity is useful if one wants to use a Faber--Schauder integral itself as an integrator, for instance in the construction of iterated integrals or in a differential equation theory driven by paths in Faber--Schauder energy spaces.

The assumption is satisfied in many natural examples where the integrand has sufficiently fast Faber--Schauder energy decay, as illustrated in the next section. There we give deterministic and stochastic examples showing the scope of the theory, including cases where the classical H\"older--Young or finite-variation Young conditions cannot be verified.

\bigskip

\section{Examples and comparison with Young theory} \label{sec:examples}

In this section, we illustrate the scope of the Faber--Schauder integral theory through deterministic and stochastic examples, including examples beyond the classical H\"older--Young and finite-variation Young regimes.

\medskip

\subsection{Deterministic examples}   \label{subsec:deterministic-examples}

We first show that the classical H\"older--Young regime is contained in the energy-space framework developed in the previous section. For $0<\alpha\le1$, write
\begin{equation}    \label{def:Holder-norm}
    [f]_{\alpha}:=\sup_{0 \le s\ne t \le 1}\frac{|f(t)-f(s)|}{|t-s|^{\alpha}}, \qquad \|f\|_{C^\alpha}:=\|f\|_\infty+[f]_\alpha,
\end{equation}
and let $C^\alpha([0,1])$ denote the space of $\alpha$-H\"older continuous functions on $[0,1]$.

\begin{proposition}[H\"older functions belong to energy spaces]
\label{prop:holder-embedding-energy}
    Let $p>1$ and $0<\alpha\le1$. If $\alpha p > 1$, then there exists a constant $C_{\alpha,p}>0$ such that
    \[
        \|f\|_{\EE^p}\le C_{\alpha,p}\|f\|_{C^\alpha}, \qquad f\in C^\alpha([0,1]).
    \]
    In particular, $C^\alpha([0,1])\subset \EE^p$.
\end{proposition}

\begin{proof}
    Let $f\in C^\alpha([0,1])$. For $n\ge0$ and $k\in I_n$, the coefficient formula \eqref{eq:FS-coeff} gives
    \begin{align}
        \big|\theta_{n,k}^f\big|
        &= 2^{\frac n2}\big|2f(t_{n+1,2k+1})-f(t_{n,k})-f(t_{n,k+1})\big|   \label{ineq:FS-coefficient-bound-Holder}
        \\
        &\le 2^{\frac n2} \Big( \big|f(t_{n+1,2k+1})-f(t_{n,k})\big| + \big|f(t_{n+1,2k+1})-f(t_{n,k+1})\big| \Big) \le 2^{1-\alpha}[f]_\alpha 2^{n(\frac12-\alpha)}.   \nonumber
    \end{align}
    Therefore, for $n\ge0$,
    \[
        \zetaen{f}{p}(n) = \bigg(2^{-\frac{np}{2}}\sum_{k\in I_n}\big|\theta_{n,k}^f\big|^p\bigg)^{\frac1p}
        \le 2^{1-\alpha}[f]_\alpha 2^{-n(\alpha-\frac1p)}.
    \]
    Since $\alpha>1/p$, this gives $\|\zetaen{f}{p}\|_{\ell^p(\N_0)}<\infty$. The affine level is also finite, since
    \[
        \zetaen{f}{p}(-1)=2^{1-\frac1p}\big|f(1)-f(0)\big|\le 2^{1-\frac1p}[f]_\alpha.
    \]
    Hence, $\zetaen{f}{p}\in\ell^p(\N_{-1})$.

    It remains to estimate the Schauder tail. From the same coefficient bound,
    \[
        2^{-\frac n2}\sup_{k\in I_n}\big|\theta_{n,k}^f\big| \le 2^{1-\alpha}[f]_\alpha 2^{-\alpha n}.
    \]
    Thus
    \[
        H_f(n) = \sum_{m=n}^{\infty}2^{-\frac m2}\sup_{k\in I_m}\big|\theta_{m,k}^f\big|
        \le 2^{1-\alpha}[f]_\alpha\sum_{m=n}^{\infty}2^{-\alpha m} = \frac{2^{1-\alpha}}{1-2^{-\alpha}}[f]_\alpha 2^{-\alpha n}.
    \]
    Hence $H_f\in\ell^p(\N_0)$. Combining these estimates with Definition~\ref{def:energy-spaces} proves the claimed bound.
\end{proof}

We now verify that, in the classical H\"older--Young regime, the Faber--Schauder integral constructed in this paper agrees with the usual Young integral. We write
\[
    \int_0^t f\,d^Yg
\]
for the classical Young integral.

\begin{theorem}[H\"older--Young regime]
\label{thm:holder-young-regime}
    Let $0<\alpha,\beta\le1$ satisfy $\alpha+\beta>1$. If $f\in C^\alpha([0,1])$ and $g\in C^\beta([0,1])$, then there exist conjugate exponents $p,q>1$ such that
    \[
        f\in\EE^p, \qquad g\in\EE^q.
    \]
    Consequently, the Faber--Schauder integrals $I_{f,g}$ and $I_{g,f}$ exist on $[0,1]$ and agree with the classical Young integrals:
    \[
        I_{f,g}(t)=\int_0^t f\,d^Yg, \qquad I_{g,f}(t)=\int_0^t g\,d^Yf, \qquad t\in[0,1].
    \]
    Moreover, the integration-by-parts identity of Theorem~\ref{thm:extended-integration-by-parts} holds, and the dyadic Young--Lo\`eve estimate of Theorem~\ref{thm:dyadic-young-loeve} applies.
\end{theorem}

\begin{proof}
    Since $\alpha+\beta>1$, we may choose $p>1$ such that
    \[
        1-\beta<\frac1p<\alpha.
    \]
    Let $q$ be the conjugate exponent of $p$. Then
    \[
        \frac1q=1-\frac1p<\beta.
    \]
    Hence $\alpha p > 1$ and $\beta q > 1$. Proposition~\ref{prop:holder-embedding-energy} gives $f\in\EE^p$, $g\in\EE^q$, and Corollary~\ref{cor:energy-space-full-theory} proves the existence of Faber--Schauder integrals $I_{f,g}$ and $I_{g, f}$ on $[0,1]$, along with the integration-by-parts identity and the dyadic Young--Lo\`eve estimate.

    To show that the Faber--Schauder integrals agree with the Young integrals, let $t\in\T$. For every $n\ge\ell(t)$, the sum $S_n^L(t;f,g)$ is a left Riemann sum over the dyadic partition of $[0,t]$, with mesh size $2^{-n}$. Since $\alpha+\beta>1$, the classical Young integral exists and these Riemann sums converge to it as $n\to\infty$. Therefore
    \[
        I_{f,g}^{\T}(t)=\int_0^t f\,d^Yg, \qquad t\in\T.
    \]
    Both sides are continuous functions of $t$ on $[0,1]$: the left-hand side by construction of $I_{f,g}$, and the right-hand side by the classical Young theory. Since $\T$ is dense in $[0,1]$, the identity extends to all $t\in[0,1]$. The identity for $I_{g,f}$ follows by the same argument, with the roles of $f$ and $g$ interchanged.
\end{proof}

\begin{corollary}[Compatibility with the variation Young integral] \label{cor:variation-young-compatibility}
    Let $p,q>1$ satisfy
    \[
        \frac1p+\frac1q\ge1.
    \]
    Let $f\in\EE^p$ and $g\in\EE^q$. Suppose, in addition, that there exist $r,s\ge1$ such that $f$ has finite $r$-variation, $g$ has finite $s$-variation, and
    \[
        \frac1r+\frac1s>1.
    \]
    Then the classical Young integrals $\int_0^t f\,d^Yg$ and $\int_0^t g\,d^Yf$ exist, and
    \[
        I_{f,g}(t)=\int_0^t f\,d^Yg, \qquad I_{g,f}(t)=\int_0^t g\,d^Yf, \qquad t\in[0,1].
    \]
\end{corollary}

\begin{proof}
    The existence and continuity of the Faber--Schauder integrals follow from Corollary~\ref{cor:energy-space-full-theory}, while the existence and continuity of the Young integrals follow from the classical finite-variation Young theorem. For each $t\in\T$, the dyadic left sums $S_n^L(t;f,g)$ are Riemann--Stieltjes sums with mesh tending to zero; hence they converge both to $I_{f,g}^{\T}(t)$ by definition and to $\int_0^t f\,d^Yg$ by Young's theorem. Therefore the two integrals agree on $\T$, and by continuity on all of $[0,1]$. The argument for $I_{g,f}$ is the same.
\end{proof}

Theorem~\ref{thm:holder-young-regime} shows that the Faber--Schauder construction is consistent with the classical Young integral in the H\"older--Young regime, where the energy-space conditions follow from H\"older regularity. Corollary~\ref{cor:variation-young-compatibility} gives the corresponding compatibility statement on the overlap with the finite-variation Young regime. We note, however, that the energy-space assumptions are not simply finite variation assumptions in another form: they impose summability across dyadic Faber--Schauder levels. This is what allows the Faber--Schauder integral to exist also in examples where neither the H\"older--Young nor the finite-variation Young criterion applies.

We next give deterministic examples showing that the coefficient-side condition $f\in\EE^p$, $g\in\EE^q$ can hold even when the classical H\"older--Young condition cannot be verified. The key point is that the level-energy component of the $\EE^p$ norm is sensitive to how the Faber--Schauder coefficients are distributed across each level, whereas H\"older regularity is governed by the largest coefficient size on each level.

\begin{proposition}[Level-sparse Faber--Schauder series]    \label{prop:level-sparse-fs-membership}
    Let $p>1$. For each $n\ge0$, let $\Lambda_n\subset I_n$ be a set of active indices, and let $(a_{n,k})_{k\in\Lambda_n}$ be real coefficients. Define
    \[
        A_n:=\sup_{k\in\Lambda_n}|a_{n,k}|, \qquad M_n:=\#\Lambda_n,
    \]
    where $\#\Lambda_n$ denotes the cardinality of $\Lambda_n$, with the convention that $A_n=0$ if $\Lambda_n=\emptyset$. Suppose that there exist constants $C_A,C_M>0$, $\alpha>0$, and $\eta\in[0,1]$ such that $p\alpha>\eta$ and
    \[
        A_n\le C_A \, 2^{n(\frac12-\alpha)}, \qquad M_n\le C_M \, 2^{\eta n}, \qquad n\ge0.
    \]
    Then, for every $c,b\in\R$, the function
    \[
        f(t):=c+bt+\sum_{n=0}^{\infty} \sum_{k\in\Lambda_n}a_{n,k}e_{n,k}(t), \qquad t\in[0,1],
    \]
    is well-defined and belongs to $\EE^p$.
\end{proposition}

\begin{proof}
    We first check uniform convergence. At each point $t\in[0,1]$, at most one Faber--Schauder function at level $n$ is nonzero, and its height is at most $2^{-\frac n2-1}$. Hence
    \[
        \bigg\|\sum_{k\in\Lambda_n}a_{n,k}e_{n,k}\bigg\|_{\infty}
        \le \frac12\,2^{-\frac n2}A_n \le \frac{C_A}{2}2^{-\alpha n}.
    \]
    Since $\alpha>0$, the series converges uniformly on $[0,1]$, and therefore $f\in\Cn$.

    By construction, the affine coefficient is $\theta_{-1,0}^f=b$, and for $n\ge0$
    \[
        \theta_{n,k}^f=
        \begin{cases}
            a_{n,k}, & k\in\Lambda_n,\\
            0, & k\notin\Lambda_n.
        \end{cases}
    \]
    The affine level contributes only a finite quantity
    \[
        \zetaen{f}{p}(-1)=2^{1-\frac1p}|b|.
    \]
    For $n\ge0$, we have
    \[
        \big(\zetaen{f}{p}(n)\big)^p
        = 2^{-\frac{np}{2}}\sum_{k\in I_n}\big|\theta_{n,k}^f\big|^p
        = 2^{-\frac{np}{2}}\sum_{k\in\Lambda_n}|a_{n,k}|^p
        \le 2^{-\frac{np}{2}}M_nA_n^p.
    \]
    Using the assumptions on $A_n$ and $M_n$, we obtain
    \[
        \big(\zetaen{f}{p}(n)\big)^p \le C_MC_A^p\,2^{(\eta-p\alpha)n}.
    \]
    Since $p\alpha>\eta$, this implies $\zetaen{f}{p} \in \ell^p(\N_{-1})$.

    It remains to check the Schauder tail condition. For $n\ge0$,
    \[
        H_f(n) = \sum_{m=n}^{\infty}2^{-\frac m2}\sup_{k\in I_m}\big|\theta_{m,k}^f\big|
        \le \sum_{m=n}^{\infty}2^{-\frac m2}A_m
        \le C_A\sum_{m=n}^{\infty}2^{-\alpha m}
        = \frac{C_A}{1-2^{-\alpha}}2^{-\alpha n}.
    \]
    Hence $H_f\in\ell^p(\N_0)$. Therefore $f\in\EE^p$.
\end{proof}

The preceding proposition gives many paths in $\EE^p$ whose roughness is distributed over a controlled number of Faber--Schauder coefficients at each level. Using these paths, we can construct examples beyond the scope of the H\"older--Young regime, while still remaining within the Faber--Schauder integral theory.

\begin{example}[A Faber--Schauder class beyond the H\"older--Young regime]
\label{ex:sparse-beyond-young}
    Let $0<\alpha,\beta<1$ satisfy $\alpha+\beta\le1$. Choose conjugate exponents $p,q>1$ and parameters $\eta_f,\eta_g\in[0,1]$ such that
    \[
        p\alpha>\eta_f, \qquad q\beta>\eta_g.
    \]
    For each $n\ge0$, let $\Lambda_n^f,\Lambda_n^g\subset I_n$ be nonempty active sets with cardinalities satisfying
    \[
        \#\Lambda_n^f\le C_f2^{\eta_f n}, \qquad \#\Lambda_n^g\le C_g2^{\eta_g n}
    \]
    for some constants $C_f,C_g>0$. Choose signs $\sigma^f_{n,k},\sigma^g_{n,k}\in\{-1,1\}$ and define
    \[
        f(t):=\sum_{n=0}^{\infty}\sum_{k\in\Lambda_n^f}\sigma^f_{n,k}2^{n(\frac12-\alpha)}e_{n,k}(t), \quad g(t):=\sum_{n=0}^{\infty}\sum_{k\in\Lambda_n^g}\sigma^g_{n,k}2^{n(\frac12-\beta)}e_{n,k}(t), \quad t \in [0, 1].
    \]
    Then, $f\in\EE^p$ and $g\in\EE^q$. Consequently, the Faber--Schauder integrals $I_{f,g}$ and $I_{g,f}$ exist on $[0,1]$; moreover, integration by parts and the dyadic Young--Lo\`eve estimate hold.

    On the other hand, $f$ is not $\gamma$-H\"older continuous for any $\gamma>\alpha$, and $g$ is not $\delta$-H\"older continuous for any $\delta>\beta$. Hence, there are no H\"older exponents for this pair whose sum is strictly larger than one.
\end{example}

\begin{proof}
    We first prove that $f\in\EE^p$. For each $n\ge0$, the active coefficients of $f$ satisfy
    \[
        \sup_{k\in\Lambda_n^f}\big|\sigma^f_{n,k}2^{n(\frac12-\alpha)}\big|=2^{n(\frac12-\alpha)}, \qquad \#\Lambda_n^f\le C_f2^{\eta_f n}.
    \]
    Proposition~\ref{prop:level-sparse-fs-membership}, applied with $A_n=2^{n(\frac12-\alpha)}$ and $M_n=\#\Lambda_n^f$, gives $f\in\EE^p$, since $p\alpha>\eta_f$. The same argument, with $\beta$, $\eta_g$, $q$, and $\Lambda_n^g$ in place of $\alpha$, $\eta_f$, $p$, and $\Lambda_n^f$, gives $g\in\EE^q$. Therefore, the conclusions about the Faber--Schauder integrals follow from Corollary~\ref{cor:energy-space-full-theory}.

    We next show that $f$ has no H\"older regularity exponent strictly larger than $\alpha$. Suppose, for contradiction, that $f\in C^\gamma([0,1])$ for some $\gamma>\alpha$. Then, the coefficient estimate \eqref{ineq:FS-coefficient-bound-Holder} in the proof of Proposition~\ref{prop:holder-embedding-energy} gives
    \[
        \big|\theta^f_{n,k}\big|\le C_\gamma2^{n(\frac12-\gamma)}, \qquad n\ge0, \quad k\in I_n,
    \]
    for some constant $C_\gamma>0$. Since $\Lambda_n^f$ is nonempty for every $n$, choose $k_n\in\Lambda_n^f$. Then
    \[
        2^{n(\frac12-\alpha)}=\big|\theta^f_{n,k_n}\big|\le C_\gamma2^{n(\frac12-\gamma)}.
    \]
    Equivalently, $2^{n(\gamma-\alpha)}\le C_\gamma$ for all $n\ge0$, which is impossible because $\gamma>\alpha$. Hence, $f$ is not $\gamma$-H\"older continuous for any $\gamma>\alpha$. The proof for $g$ is identical.

    If the classical H\"older--Young condition were verifiable for this pair $(f, g)$, then there would exist exponents $\gamma,\delta>0$ such that
    \[
        f\in C^\gamma([0,1]), \qquad g\in C^\delta([0,1]), \qquad \gamma+\delta>1.
    \]
    The preceding paragraph forces $\gamma\le\alpha$ and $\delta\le\beta$, hence $\gamma+\delta\le\alpha+\beta\le1$, a contradiction.
\end{proof}

This example shows that the dyadic coefficient condition is not merely a reformulation of the H\"older--Young condition; it can apply to pairs for which no H\"older exponents satisfying the classical Young threshold are available. The trade-off
\[
    p\alpha>\eta_f, \qquad q\beta>\eta_g
\]
shows that the level energies can remain summable even when the number of active Faber--Schauder coefficients grows exponentially with the level, provided that this growth is sufficiently slow relative to the coefficient decay.

We now construct a pair of paths for which the classical Young theorem based on the finite-variation condition cannot be applied, while the Faber--Schauder integral theory developed here still applies. In contrast to the level-sparse example in the previous subsection, the following construction activates all Faber--Schauder coefficients at each level. We use the coefficient characterization \eqref{con:equiv-zetaen-p-var} of infinite $p$-variation.

\begin{example}[A Faber--Schauder pair beyond the variation Young regime]   \label{ex:beyond-variation-young}
    Let $p,q>1$ be conjugates. Define
    \[
        a_n:=2^{-\frac np}(n+1)^{-\frac2p}, \qquad b_n:=2^{-\frac nq}(n+1)^{-\frac2q}, \qquad n\ge0,
    \]
    and set
    \[
        f(t):=\sum_{n=0}^{\infty}\sum_{k\in I_n}2^{\frac n2}a_n e_{n,k}(t), \qquad
        g(t):=\sum_{n=0}^{\infty}\sum_{k\in I_n}2^{\frac n2}b_n e_{n,k}(t).
    \]
    Then $f\in\EE^p$ and $g\in\EE^q$. Consequently, the Faber--Schauder integrals $I_{f,g}$ and $I_{g,f}$ exist on $[0,1]$; moreover, the integration-by-parts identity and the dyadic Young--Lo\`eve estimate hold.

    On the other hand, $f$ has infinite $r$-variation for every $1\le r<p$, and $g$ has infinite $s$-variation for every $1\le s<q$. Hence, there are no finite-variation exponents $r,s\ge1$ for this pair satisfying
    \[
        \frac1r+\frac1s>1.
    \]
    Thus, the finite-variation Young condition cannot be satisfied by this pair $(f, g)$.
\end{example}

\begin{proof}
    The uniform convergence of the two series and the energy-space membership $f\in\EE^p$, $g\in\EE^q$ follow from the same computation as in the proof of Proposition~\ref{prop:energy-space-monotonicity}. The conclusions about the Faber--Schauder integrals follow from Corollary~\ref{cor:energy-space-full-theory}.

    We now prove that the finite-variation Young condition cannot be satisfied by this pair. Let $1<r<p$. Then
    \[
        \big(\zetaen{f}{r}(n)\big)^r
        = 2^{-\frac{nr}{2}}\sum_{k\in I_n}\big|\theta^f_{n,k}\big|^r
        = 2^n a_n^r
        = 2^{n(1-\frac rp)}(n+1)^{-\frac{2r}{p}} \xrightarrow{n\to\infty}\infty.
    \]
    Applying the coefficient characterization \eqref{con:equiv-zetaen-p-var} with exponent $r$, we get $\|f\|_{r\text{-var}}=\infty$. If $f$ had finite $1$-variation, then, since $f$ is bounded, it would have finite $r$-variation for every $r>1$, contradicting the preceding conclusion for any $1<r<p$. Thus, $f$ has infinite $r$-variation for every $1\le r<p$. Similarly, $g$ has infinite $s$-variation for every $1\le s<q$.

    Suppose now that the finite-variation Young's theorem were applicable to this pair. Then there would exist $r,s\ge1$ such that $f$ has finite $r$-variation, $g$ has finite $s$-variation, and
    \[
        \frac1r+\frac1s>1.
    \]
    The variation divergence proved above forces $r\ge p$ and $s\ge q$. Therefore
    \[
        \frac1r+\frac1s\le \frac1p+\frac1q=1,
    \]
    a contradiction, and no such variation exponents exist.
\end{proof}

The particular choice of the polynomial weights $(n+1)^{-\frac2p}$ and $(n+1)^{-\frac2q}$ in Example~\ref{ex:beyond-variation-young} is not essential. More generally, for fixed conjugate exponents $p,q>1$, the same argument applies to any positive sequences $(a_n)_{n\ge0}$ and $(b_n)_{n\ge0}$ satisfying
\[
    \sum_{n=0}^{\infty}2^n a_n^p<\infty, \quad \sum_{n=0}^{\infty}\Big(\sum_{m=n}^{\infty}a_m\Big)^p<\infty, \quad \sum_{n=0}^{\infty}2^n b_n^q<\infty, \quad \sum_{n=0}^{\infty}\Big(\sum_{m=n}^{\infty}b_m\Big)^q<\infty,
\]
and
\[
    \limsup_{n\to\infty}2^n a_n^r=\infty \quad\text{for every }1<r<p, \qquad \limsup_{n\to\infty}2^n b_n^s=\infty \quad\text{for every }1<s<q.
\]
Then the corresponding Faber--Schauder series
\[
    f(t)=\sum_{n=0}^{\infty}\sum_{k\in I_n}2^{\frac n2}a_ne_{n,k}(t), \qquad g(t)=\sum_{n=0}^{\infty}\sum_{k\in I_n}2^{\frac n2}b_ne_{n,k}(t)
\]
satisfy $f\in\EE^p$ and $g\in\EE^q$, while the pair cannot satisfy the finite-variation assumptions of the classical Young theorem.

\medskip

\subsection{Stochastic examples}   \label{subsec:stochastic-examples}

All the constructions above involved deterministic functions, but the same coefficient-side method also yields pathwise integrals for suitable stochastic processes. For instance, the last Faber--Schauder construction in Example~\ref{ex:beyond-variation-young} has the following stochastic version.

\begin{example}[Stochastic pairs beyond the variation Young regime]    \label{ex:stochastic-beyond-variation-young}
    Let $p,q>1$ be conjugates, and recall
    \[
        a_n:=2^{-\frac np}(n+1)^{-\frac2p}, \qquad
        b_n:=2^{-\frac nq}(n+1)^{-\frac2q}.
    \]
    Let $(\xi_{n,k})_{n\ge0,k\in I_n}$ and $(\eta_{n,k})_{n\ge0,k\in I_n}$ be two independent families of independent standard Gaussian random variables, and define
    \[
        F(t):=\sum_{n=0}^{\infty}\sum_{k\in I_n}2^{\frac n2}a_n\xi_{n,k}e_{n,k}(t), \qquad
        G(t):=\sum_{n=0}^{\infty}\sum_{k\in I_n}2^{\frac n2}b_n\eta_{n,k}e_{n,k}(t), \qquad t \in [0,1].
    \]
    Then $F$ and $G$ are well-defined continuous centered Gaussian processes, and $F\in\EE^p$, $G\in\EE^q$ almost surely. Consequently, the Faber--Schauder integrals $I_{F,G}$ and $I_{G,F}$ exist on $[0,1]$ almost surely; moreover, integration by parts and the dyadic Young--Lo\`eve estimate hold.

    On the other hand, almost surely, there are no exponents $r,s\ge1$ such that $F$ has finite $r$-variation, $G$ has finite $s$-variation, and $1/r+1/s>1$.
\end{example}

\begin{proof}   
    For $F$, we have
    \[
        \theta^F_{n,k}=2^{\frac n2}a_n\xi_{n,k}, \qquad n\ge0, \quad k\in I_n,
    \]
    and hence
    \[
        \mathbb E\Big[\big(\zetaen{F}{p}(n)\big)^p\Big]
        = a_n^p\sum_{k\in I_n}\mathbb E|\xi_{n,k}|^p
        = \mathbb E|\xi_{0,0}|^p\,2^na_n^p
        = \mathbb E|\xi_{0,0}|^p\,(n+1)^{-2}.
    \]
    The affine coefficient vanishes, since all $e_{n,k}$ with $n\ge0$ vanish at $0$ and $1$. Therefore, by Fubini's theorem, $\zetaen{F}{p}\in\ell^p(\N_{-1})$ almost surely. For the Schauder tail, the Gaussian tail bound and the union bound give
    \[
        \mathbb P\bigg(\max_{k\in I_n}|\xi_{n,k}|>C\sqrt n\bigg)
        \le 2^{n+1}e^{-\frac{C^2}{2}n}.
    \]
    Choosing $C>0$ large enough that $C^2>2\log2$, we obtain
    \[
        \sum_{n=1}^{\infty}\mathbb P\bigg(\max_{k\in I_n}|\xi_{n,k}|>C\sqrt n\bigg)<\infty.
    \]
    By the Borel--Cantelli lemma, almost surely, for all sufficiently large $n$,
    \[
        \max_{k\in I_n}|\xi_{n,k}|\le C\sqrt n.
    \]
    Thus, almost surely, for all sufficiently large $n$,
    \begin{equation}    \label{bound-stochastic-full}
        2^{-\frac n2}\sup_{k\in I_n}\big|\theta^F_{n,k}\big|
        = a_n\max_{k\in I_n}|\xi_{n,k}|
        \le C\sqrt n\,2^{-\frac np}(n+1)^{-\frac2p}.
    \end{equation}
    This bound implies the uniform convergence of the defining series, since at each point at most one Faber--Schauder function at level $n$ is nonzero and has height at most $2^{-n/2-1}$. Hence, $F$ has continuous sample paths. Moreover, if $F_N$ denotes the level-$N$ partial sum, then
    \[
        \mathbb E\big|F(t)-F_N(t)\big|^2
        \le \frac14\sum_{n=N+1}^{\infty}a_n^2
        \xrightarrow{N\to\infty}0 \qquad \text{for each fixed } t \in [0, 1].
    \]
    Thus, for every finite collection of times, the corresponding random vectors of $F_N$ converge in $L^2$ to those of $F$. Since each $F_N$ is a centered Gaussian process, the limit $F$ is also a centered Gaussian process. From the bound \eqref{bound-stochastic-full}, we also have
    \[
        H_F(n) = \sum_{m=n}^{\infty}2^{-\frac m2}\sup_{k\in I_m}|\theta_{m,k}^F|
        \le C\sum_{m=n}^{\infty}\sqrt m\,2^{-\frac mp}(m+1)^{-\frac2p}
        \le C_p\sqrt n\,2^{-\frac np}(n+1)^{-\frac2p}
    \]
    for all sufficiently large $n$, and therefore $H_F\in\ell^p(\N_0)$ almost surely. Hence, $F\in\EE^p$ almost surely. The proof for $G\in\EE^q$ is identical, and the same argument shows that $G$ is a continuous centered Gaussian process. Corollary~\ref{cor:energy-space-full-theory} implies that the Faber--Schauder integrals $I_{F,G}$ and $I_{G,F}$ exist on $[0,1]$ almost surely.
    
    Moreover, this random pair is still outside the finite-variation Young regime. For every $1<r<p$,
    \[
        2^{-\frac{nr}{2}}\sum_{k\in I_n} \big|\theta^F_{n,k}\big|^r
        = a_n^r\sum_{k\in I_n}|\xi_{n,k}|^r.
    \]
    For each $n$, the variables $(|\xi_{n,k}|^r)_{k\in I_n}$ are independent and identically distributed with finite second moment. Hence,
    \[
        \mathbb{E}\Big[ 2^{-n}\sum_{k\in I_n} |\xi_{n,k}|^r \Big] = \mathbb{E}|\xi_{0,0}|^r, \qquad \text{Var}\Big( 2^{-n}\sum_{k\in I_n} |\xi_{n,k}|^r \Big) = 2^{-n}\text{Var}( |\xi_{0,0}|^r).
    \]
    By Chebyshev's inequality,
    \[
        \sum_{n=0}^{\infty}
        \mathbb{P}\bigg( \Big| 2^{-n}\sum_{k\in I_n}|\xi_{n,k}|^r- \mathbb{E}|\xi_{0,0}|^r \Big|>\varepsilon \bigg)
        \le \varepsilon^{-2} \text{Var}(|\xi_{0,0}|^r) \sum_{n=0}^{\infty}2^{-n} < \infty.
    \]
    The Borel--Cantelli lemma yields
    \[
        2^{-n}\sum_{k\in I_n}|\xi_{n,k}|^r \longrightarrow \mathbb{E}|\xi_{0,0}|^r \quad\text{almost surely},
    \]
    hence, almost surely, for all sufficiently large $n$
    \[
        2^{-\frac{nr}{2}}\sum_{k\in I_n}\big|\theta^F_{n,k}\big|^r
        \ge c_r 2^n a_n^r = c_r2^{n(1-\frac rp)}(n+1)^{-\frac{2r}{p}} \xrightarrow{n\to\infty}\infty
    \]
    for some constant $c_r>0$. Applying \eqref{con:equiv-zetaen-p-var} with this fixed exponent $r$ gives $\|F\|_{r\text{-var}}=\infty$ almost surely. Taking the intersection over all rational $r\in(1,p)$ and using the monotonicity of variation exponents, we obtain that $F$ has infinite $r$-variation for every $1\le r<p$ almost surely. Similarly, $G$ has infinite $s$-variation for every $1\le s<q$. Therefore, there are no exponents $r,s\ge1$ such that $F$ has finite $r$-variation, $G$ has finite $s$-variation, and
    \[
        \frac1r+\frac1s>1.
    \]
    Thus, the finite-variation Young theorem cannot be applied to this pair.
\end{proof}

We now study the energy-space criterion on fractional Brownian motion~(fBM). Let $B^H=(B^H_t)_{t\in[0,1]}$ be a fBM with Hurst index $H\in(0,1)$, that is, a centered Gaussian process with covariance
\begin{equation}    \label{eq:fBM-cov}
    \mathbb E[B^H_sB^H_t]=\frac12\big(s^{2H}+t^{2H}-|t-s|^{2H}\big), \qquad s,t\in[0,1].
\end{equation}

\begin{lemma}[Faber--Schauder coefficients of fBM]  \label{lem:fbm-fs-coefficients}
    For every $n\ge0$ and $k\in I_n$,
    \[
        \theta^{B^H}_{n,k} = 2^{\frac n2}\big(2B^H_{t_{n+1,2k+1}}-B^H_{t_{n,k}}-B^H_{t_{n,k+1}}\big)
    \]
    is a centered Gaussian random variable. Moreover, for every $r>0$ there exists a constant $C_{H,r}>0$ such that
    \[
        \mathbb E\big|\theta^{B^H}_{n,k}\big|^r = C_{H,r}2^{nr(\frac12-H)}, \qquad n\ge0, \quad k\in I_n.
    \]
\end{lemma}

\begin{proof}
    By stationary increments and self-similarity of fBM,
    \[
        2B^H_{t_{n+1,2k+1}}-B^H_{t_{n,k}}-B^H_{t_{n,k+1}}
        \stackrel{d}{=}
        2^{-nH}\big(2B^H_{1/2}-B^H_0-B^H_1\big).
    \]
    Multiplying by $2^{n/2}$ gives
    \begin{equation}    \label{eq:fBM-increments-dist}
        \theta^{B^H}_{n,k}\stackrel{d}{=}2^{n(\frac12-H)}\big(2B^H_{1/2}-B^H_0-B^H_1\big).
    \end{equation}
    The random variable on the right-hand side is centered Gaussian and nondegenerate for $H\in(0,1)$. Taking absolute moments and setting $C_{H,r} := \mathbb E |2B^H_{1/2}-B^H_0-B^H_1|^r$ proves the claim.
\end{proof}

Note that when $r = 2$, we have $C_{H,2}=2^{2-2H}-1$, in agreement with the corresponding variance formula in \cite[(6.6)]{fake_fBM}.

The following result identifies the sharp criterion for membership of fBM in the Faber--Schauder energy spaces.

\begin{theorem}[Sharp energy threshold for fBM]   \label{thm:fbm-energy-space}
    Let $p>1$ and $H\in(0,1)$.  Then
    \[
        \mathbb{P}(B^H\in \EE^p) =
        \begin{cases}
        1, & H>1/p,\\
        0, & H\le 1/p.
        \end{cases}
    \]
    Equivalently, $B^H\in \EE^p$ almost surely if and only if $H>1/p$.
\end{theorem}

\begin{proof}
    We work with the continuous modification of $B^H$, which exists for every $H\in(0,1)$.
    
    We first assume $H > 1/p$. The affine coefficient $\theta^{B^H}_{-1,0}=B^H_1-B^H_0$ is finite almost surely. For $n \ge 0$, Lemma~\ref{lem:fbm-fs-coefficients} gives
    \[
        \mathbb E\Big[\big(\zetaen{B^H}{p}(n)\big)^p\Big] = 2^{-\frac{np}{2}}\sum_{k\in I_n}\mathbb E \big|\theta^{B^H}_{n,k}\big|^p = c_{H,p}2^{n(1-pH)}.
    \]
    Since $pH>1$, we have
    \[
        \sum_{n=0}^{\infty}\mathbb E\Big[\big(\zetaen{B^H}{p}(n)\big)^p\Big] < \infty.
    \]
    Hence, by Fubini's theorem, $\zetaen{B^H}{p}\in\ell^p(\N_{-1})$ almost surely.
    
    For the Schauder tail condition, define
    \begin{equation}    \label{def:Znk}
        Z_{n,k}:=2^{-n(\frac12-H)}\theta^{B^H}_{n,k}, \qquad n\ge0, \quad k\in I_n.
    \end{equation}
    By \eqref{eq:fBM-increments-dist}, the random variables $Z_{n,k}$ are centered Gaussian with variance independent of $n$ and $k$. Hence, by the Gaussian tail bound and the union bound, one can choose $C>0$ such that
    \[
        \sum_{n=1}^{\infty} \mathbb P \Big(\max_{k\in I_n}|Z_{n,k}|>C\sqrt n\Big) < \infty.
    \]
    By the Borel--Cantelli lemma, almost surely, for all sufficiently large $n$,
    \[
        \max_{k\in I_n}|Z_{n,k}|\le C\sqrt n, \qquad \text{equivalently}, \qquad 2^{-\frac n2}\sup_{k\in I_n}\big|\theta^{B^H}_{n,k}\big|
        \le C\sqrt n\,2^{-Hn}.
    \]
    Hence, almost surely, for all sufficiently large $n$,
    \[
        H_{B^H}(n)
        = \sum_{m=n}^{\infty}2^{-\frac m2}\sup_{k\in I_m}\big|\theta^{B^H}_{m,k}\big|
        \le C\sum_{m=n}^{\infty}\sqrt m\,2^{-Hm}
        \le C_H\sqrt n\,2^{-Hn},
    \]
    where $C_H>0$ is a constant depending only on $H$ and on $C$. Since $(\sqrt n\,2^{-Hn})_{n\ge0}\in\ell^p(\N_0)$ for every $p>1$, the tail of $H_{B^H}$ belongs to $\ell^p(\N_0)$ almost surely. The remaining finitely many values are finite almost surely, hence $H_{B^H}\in\ell^p(\N_0)$ almost surely. Therefore, $B^H\in\EE^p$ almost surely.

    Conversely, suppose $H\le 1/p$. We will show that $\zeta^{(p)}_{B^H}\notin \ell^p(\mathbb{N}_{-1})$ almost surely, as this implies $B^H\notin \EE^p$ almost surely. For $n\ge0$ and $k\in I_n$, we recall
    \[
        Z_{n,k}:=2^{-n(\frac12-H)}\theta^{B^H}_{n,k} = 2^{nH}\big(2B^H_{t_{n+1,2k+1}}-B^H_{t_{n,k}}-B^H_{t_{n,k+1}}\big).
    \]
    By stationary increments and self-similarity of fBM, for each fixed $n$,
    \begin{equation}    \label{eq:row-wise-dist}
        (Z_{n,0},\ldots,Z_{n,2^n-1}) \overset{d}{=} (Y_0,\ldots,Y_{2^n-1}),
    \end{equation}
    where
    \[
        Y_k:=2\widetilde B^H_{k+1/2}-\widetilde B^H_k-\widetilde B^H_{k+1},
    \]
    and $\widetilde B^H$ is an fBM with the same Hurst index, defined on $[0,\infty)$. Thus $(Y_k)_{k\ge0}$ is a centered stationary Gaussian sequence.
    
    We first estimate its covariance. Let $\rho(m):=\mathbb{E}[Y_0Y_m]$ for $m\ge0$. Define
    \[
        X_j:=\widetilde B^H_{(j+1)/2}-\widetilde B^H_{j/2}, \qquad j\ge0.
    \]
    Then $Y_k=X_{2k}-X_{2k+1}$. If $\gamma(r):=\mathbb{E}[X_0X_r]$, then from \eqref{eq:fBM-cov}
    \[
        \gamma(r)=2^{-2H-1}\left(|r+1|^{2H}+|r-1|^{2H}-2|r|^{2H}\right), \qquad r\ge0.
    \]
    Therefore, for $m\ge1$,
    \[
        \rho(m) = \mathbb{E}\big[ (X_0-X_1)(X_{2m}-X_{2m+1}) \big] = 2\gamma(2m)-\gamma(2m-1)-\gamma(2m+1).
    \]
    For $m\ge2$, substituting the explicit formula for $\gamma$ gives
    \[
        \rho(m)=2^{-2H-1}\left(-(2m+2)^{2H}+4(2m+1)^{2H}-6(2m)^{2H}+4(2m-1)^{2H}-(2m-2)^{2H}\right).
    \]
    Writing $u(x)=x^{2H}$ for $x>0$, the expression in parentheses is $-\Delta^4u(2m-2)$, where
    \[
        \Delta^4u(x):=u(x+4)-4u(x+3)+6u(x+2)-4u(x+1)+u(x).
    \]
    Let $u^{(4)}$ denote the fourth derivative of $u$. Since
    \[
        u^{(4)}(x)=2H(2H-1)(2H-2)(2H-3)x^{2H-4}, \qquad x>0,
    \]
    we have
    \[
        |u^{(4)}(x)|\le C_Hx^{2H-4}, \qquad x>0,
    \]
    where, here and below, $C_H>0$ denotes a constant depending only on $H$, whose value may change from line to line. Moreover, by repeated use of the fundamental theorem of calculus,
    \[
        \Delta^4u(x)=\int_{[0,1]^4}u^{(4)}(x+s_1+s_2+s_3+s_4)\,ds_1ds_2ds_3ds_4.
    \]
    Therefore, for $m\ge2$,
    \[
        |\rho(m)|\le C_H\sup_{y\in[2m-2,2m+2]} y^{2H-4}\le C_Hm^{2H-4}.
    \]
    Since $H<1$, we have $2H-4<-2$, and
    \begin{equation}    \label{ineq:rho-summable}
        \sum_{m=0}^{\infty}|\rho(m)|<\infty,
    \end{equation}
    where the finitely many terms $m=0,1$ are irrelevant for summability.
    
    We next show that the empirical averages of $|Z_{n,k}|^p$ converge almost surely. Set
    \begin{equation}    \label{def:Qn}
        Q_n:=2^{-n}\sum_{k=0}^{2^n-1}|Z_{n,k}|^p.
    \end{equation}
    By the row-wise distributional identity \eqref{eq:row-wise-dist}, we have $\mathbb{E}Q_n=\mathbb{E}|Y_0|^p$. Moreover, the covariance sequence of $(|Y_k|^p)_{k\ge0}$ is summable. Indeed, since $Y_0$ is nondegenerate, $\rho(0)>0$. For $m\ge1$, the random variables
    \[
        \frac{Y_0}{\sqrt{\rho(0)}} \quad \text{and} \quad \frac{Y_m}{\sqrt{\rho(0)}}
    \]
    form a standard bivariate normal vector with correlation $\rho(m)/\rho(0)$. By Gebelein's inequality for functions of a bivariate Gaussian vector \cite{Gebelein} (see also \cite[(1.1)]{Veraar2009}), applied to the centered square-integrable function
    \[
        x\mapsto |\rho(0)^{1/2}x|^p-\mathbb{E}|Y_0|^p,
    \]
    we obtain
    \[
        \big|\text{Cov}(|Y_0|^p,|Y_m|^p)\big| \le \frac{|\rho(m)|}{\rho(0)}\,\text{Var}(|Y_0|^p) \le C_{H,p}|\rho(m)|,
    \]
    for some constant $C_{H,p} > 0$. Here $\text{Var}(|Y_0|^p)<\infty$ because $Y_0$ is Gaussian. From \eqref{ineq:rho-summable}, we obtain
    \[
        \sum_{m=0}^{\infty}\big|\text{Cov}(|Y_0|^p,|Y_m|^p)\big|<\infty.
    \]
    By the distributional identity \eqref{eq:row-wise-dist}, the variance of $Q_n$ is equal to the variance of $2^{-n}\sum_{k=0}^{2^n-1}|Y_k|^p$. Hence
    \[
        \text{Var}(Q_n)=2^{-2n}\sum_{k,\ell=0}^{2^n-1}\text{Cov}(|Y_k|^p,|Y_\ell|^p).
    \]
    By stationarity, $\text{Cov}(|Y_k|^p,|Y_\ell|^p)$ depends only on $|\ell-k|$. More precisely, for $m\ge1$, there are $2^n-m$ pairs $(k,\ell)$ with $\ell=k+m$ and $2^n-m$ pairs with $k=\ell+m$. Hence
    \begin{align*}
        \text{Var}(Q_n) &= 2^{-2n}\bigg(2^n\text{Var}(|Y_0|^p)+2\sum_{m=1}^{2^n-1}(2^n-m)\text{Cov}(|Y_0|^p,|Y_m|^p)\bigg)
        \\
        & \le 2^{-2n}\bigg(2^n\text{Var}(|Y_0|^p)+2\sum_{m=1}^{2^n-1}(2^n-m)\big|\text{Cov}(|Y_0|^p,|Y_m|^p)\big|\bigg)
        \\
        & \le 2^{-n}\bigg(\text{Var}(|Y_0|^p)+2\sum_{m=1}^{\infty}\big|\text{Cov}(|Y_0|^p,|Y_m|^p)\big|\bigg) \le C_{H,p}\,2^{-n}.
    \end{align*}
    Here the constant $C_{H,p}$ is finite because $Y_0$ is Gaussian, so $\text{Var}(|Y_0|^p)<\infty$.
    Hence, for every $\varepsilon>0$, Chebyshev's inequality yields
    \[
        \sum_{n=0}^{\infty}\mathbb{P}\left(\left|Q_n-\mathbb{E}|Y_0|^p\right|>\varepsilon\right)\le \varepsilon^{-2}\sum_{n=0}^{\infty}\text{Var}(Q_n)<\infty.
    \]
    By the Borel--Cantelli lemma,
    \[
        Q_n\xrightarrow{n \to \infty} \mathbb{E}|Y_0|^p \quad \text{a.s.}
    \]
    Since $Y_0$ is nondegenerate Gaussian, $\mathbb{E}|Y_0|^p>0$.
    
    Finally, by the definition of the level energy in Definition~\ref{def:level-energies}, \eqref{def:Znk}, and \eqref{def:Qn}
    \[
        \Big(\zeta^{(p)}_{B^H}(n)\Big)^p = 2^{-\frac{np}{2}}\sum_{k=0}^{2^n-1}|\theta^{B^H}_{n,k}|^p=2^{n(1-pH)}Q_n.
    \]
    If $H<1/p$, then $(\zeta^{(p)}_{B^H}(n))^p\to\infty$ as $n\to\infty$ almost surely. If $H=1/p$, then
    \[
        \Big(\zeta^{(p)}_{B^H}(n)\Big)^p = Q_n\longrightarrow \mathbb{E}|Y_0|^p>0 \quad \text{a.s.}
    \]
    In both cases, the sequence $((\zeta^{(p)}_{B^H}(n))^p)_{n\ge0}$ is not summable. Therefore $\zeta^{(p)}_{B^H} \notin \ell^p(\mathbb{N}_{-1})$ and hence $B^H\notin \EE^p$ almost surely.
\end{proof}

\begin{corollary}[Pure fractional Brownian pairs]
\label{cor:pure-fbm-pairs}
    Let $B^{H_1}$ and $\widetilde B^{H_2}$ be fBMs with Hurst indices $H_1,H_2\in(0,1)$. Then the following are equivalent:
    \begin{enumerate}
        \item [(i)] $H_1+H_2>1$.
        \item [(ii)] There exist conjugates $p,q>1$ such that $B^{H_1}\in\EE^p$ and $\widetilde B^{H_2}\in\EE^q$ almost surely.
    \end{enumerate}
    In this case, the Faber--Schauder integrals $I_{B^{H_1},\widetilde B^{H_2}}$ and $I_{\widetilde B^{H_2},B^{H_1}}$ exist almost surely.
\end{corollary}

\begin{proof}
    Suppose first that $H_1+H_2>1$. Then we can choose $p>1$ such that
    \[
        1-H_2<\frac1p<H_1.
    \]
    Let $q$ be the conjugate exponent of $p$. Then
    \[
        \frac1q=1-\frac1p<H_2.
    \]
    Hence $H_1>1/p$ and $H_2>1/q$. Theorem~\ref{thm:fbm-energy-space} gives $B^{H_1}\in\EE^p$ and $\widetilde B^{H_2}\in\EE^q$ almost surely.

    Conversely, suppose that there exist conjugates $p,q>1$ such that $B^{H_1}\in\EE^p$ and $\widetilde B^{H_2}\in\EE^q$ almost surely. By the necessity part of Theorem~\ref{thm:fbm-energy-space}, we must have
    \[
        H_1>\frac1p \quad \text{and} \quad H_2>\frac1q, \qquad \text{hence} \qquad H_1+H_2>\frac1p+\frac1q=1.
    \]
    When these equivalent conditions hold, the existence of the integrals follows from Corollary~\ref{cor:energy-space-full-theory}.
\end{proof}

The preceding corollary implies that, for pure fractional Brownian pairs, the energy-space criterion is sharp and exactly matches the familiar Young threshold. The next example shows that once one component is replaced by a random level-sparse Faber--Schauder process, the coefficient-side criterion can apply outside the H\"older--Young regime.

\begin{proposition}[Random level-sparse Faber--Schauder series] \label{prop:random-level-sparse-fs}
    Let $\alpha \in (0, 1)$, $\eta\in[0,1]$, and $p>1$ satisfy $p\alpha>\eta$. For each $n\ge0$, let $\Lambda_n\subset I_n$ be a nonempty active set such that $\#\Lambda_n\le C_\Lambda 2^{\eta n}$ for some constant $C_\Lambda>0$. Let $(\xi_{n,k})_{n\ge0,\,k\in\Lambda_n}$ be independent standard Gaussian random variables, and define
    \[
        F(t) := \sum_{n=0}^{\infty} \sum_{k\in\Lambda_n} 2^{n(\frac12-\alpha)}\xi_{n,k}e_{n,k}(t), \qquad t\in[0,1].
    \]
    Then, $F$ is a well-defined continuous centered Gaussian process and $F\in\EE^p$ almost surely. Moreover, $F$ is not $\gamma$-H\"older continuous for any $\gamma\in(\alpha,1]$, almost surely.
\end{proposition}

\begin{proof}
    We first compute the level energies. For $n \ge 0$
    \[
        \theta^F_{n,k}=
        \begin{cases}
            2^{n(\frac12-\alpha)}\xi_{n,k}, & k\in\Lambda_n,\\
            ~~~0, & k\notin\Lambda_n.
        \end{cases}
    \]
    Hence
    \[
        \mathbb E\big[\big(\zetaen{F}{p}(n)\big)^p\big]
        = 2^{-\frac{np}{2}} \sum_{k\in\Lambda_n}2^{np(\frac12-\alpha)}\mathbb E|\xi_{n,k}|^p
        \le C2^{(\eta-p\alpha)n}.
    \]
    Since $p\alpha>\eta$, Fubini's theorem gives $\zetaen{F}{p}\in\ell^p(\N_{-1})$ almost surely.
    
    We next prove the Schauder tail condition. By the Gaussian tail bound and the union bound, choosing $C_0>0$ sufficiently large gives
    \[
        \sum_{n=1}^{\infty} \mathbb P\bigg( \max_{k\in\Lambda_n}|\xi_{n,k}|>C_0\sqrt n \bigg) < \infty.
    \]
    By the Borel--Cantelli lemma, almost surely, for all sufficiently large $n$,
    \[
        \max_{k\in\Lambda_n}|\xi_{n,k}|\le C_0\sqrt n, \qquad \text{hence} \qquad
        2^{-\frac n2}\sup_{k\in I_n}\big|\theta^F_{n,k}\big|
        \le C_0\sqrt n\,2^{-\alpha n}.
    \]
    Consequently,
    \[
        H_F(n) = \sum_{m=n}^{\infty}2^{-\frac m2}\sup_{k\in I_m}|\theta_{m,k}^F| \le C_0\sum_{m=n}^{\infty}\sqrt m\,2^{-\alpha m}
        \le C_\alpha\sqrt n\,2^{-\alpha n}
    \]
    for a constant $C_{\alpha} > 0$ and for all sufficiently large $n$. Since $(\sqrt n\,2^{-\alpha n})_{n\ge0}\in\ell^p(\N_0)$ and the remaining finitely many values are finite almost surely, we obtain $H_F\in\ell^p(\N_0)$ almost surely.
    
    Thus, $F\in\EE^p$ almost surely. The same estimates also imply uniform convergence of the defining Faber--Schauder series, so $F$ has continuous sample paths. Since each finite partial sum is a centered Gaussian process and the partial sums converge in $L^2$ at every fixed time, the limit $F$ is a centered Gaussian process.

    It remains to prove the sharpness of the H\"older exponent. For each $n$, choose one index $k_n\in\Lambda_n$. Since the random variables $(\xi_{n,k_n})_{n\ge0}$ are independent standard Gaussians, the second Borel--Cantelli lemma gives
    \[
        \mathbb P\big(|\xi_{n,k_n}|\ge1\ \text{infinitely often}\big)=1.
    \]
    Hence, almost surely, along infinitely many levels $n$,
    \[
        \sup_{k\in I_n}\big|\theta^F_{n,k}\big|
        \ge 2^{n(\frac12-\alpha)}.
    \]
    Fix a sample path $\omega$ in this probability-one event. If $F(\omega) \in C^\gamma([0,1])$ for some $\gamma>\alpha$, then the estimate \eqref{ineq:FS-coefficient-bound-Holder} gives
    \[
        \sup_{k\in I_n}\big|\theta^{F(\omega)}_{n,k}\big|
        \le C_\gamma(\omega)2^{n(\frac12-\gamma)}, \qquad n\ge0,
    \]
    for some finite constant $C_\gamma(\omega)>0$. This contradicts the preceding lower bound along infinitely many levels, since it would imply $2^{n(\gamma-\alpha)}\le C_\gamma(\omega)$ along infinitely many $n$. Thus, almost surely, $F$ is not $\gamma$-H\"older continuous for any $\gamma\in(\alpha,1]$.
\end{proof}

\begin{example}[A mixed Gaussian/fractional Brownian pair]  \label{ex:mixed-sparse-fbm}
    Let $H\in(0,1)$ and $\alpha\in(0,1)$ satisfy $\alpha+H\le1$. Let $\eta\in[0,1]$ satisfy $\eta(1-H)<\alpha$. For each $n\ge0$, choose a nonempty active set $\Lambda_n\subset I_n$ such that $\#\Lambda_n\le C_\Lambda 2^{\eta n}$ for some constant $C_\Lambda>0$. Let $(\xi_{n,k})_{n\ge0,\,k\in\Lambda_n}$ be independent standard Gaussian random variables, independent of $B^H$, and define
    \[
        F(t):= \sum_{n=0}^{\infty} \sum_{k\in\Lambda_n} 2^{n(\frac12-\alpha)}\xi_{n,k}e_{n,k}(t), \qquad t\in[0,1].
    \]
    Then, there exist conjugate exponents $p,q>1$ such that $F\in\EE^p$ and $B^H\in\EE^q$ almost surely. Consequently, the Faber--Schauder integrals $I_{F,B^H}$ and $I_{B^H,F}$ exist on $[0,1]$ almost surely.

    On the other hand, the classical H\"older--Young condition cannot be verified for this pair almost surely.
\end{example}

\begin{proof}
    Since $\eta(1-H)<\alpha$, we can choose $p>1$ such that
    \[
        1-H<\frac1p<\frac{\alpha}{\eta}
    \]
    when $\eta>0$. If $\eta=0$, choose any $p>1$ such that
    \[
        1-H<\frac1p<1.
    \]
    Let $q$ be the conjugate exponent of $p$. Then
    \[
        \frac1q=1-\frac1p<H.
    \]
    Also, by the choice of $p$, we have $p\alpha>\eta$.

    Hence, Proposition~\ref{prop:random-level-sparse-fs} gives $F\in\EE^p$ almost surely and Theorem~\ref{thm:fbm-energy-space} implies $B^H\in\EE^q$ almost surely. Therefore, the Faber--Schauder integral conclusions follow from Corollary~\ref{cor:energy-space-full-theory}.

    Finally, we show that the classical H\"older--Young condition cannot be verified. Proposition~\ref{prop:random-level-sparse-fs} implies, almost surely, that $F$ is not $\gamma$-H\"older continuous for any $\gamma \in (\alpha, 1]$. Also, almost surely, fBM $B^H$ is not $\delta$-H\"older continuous for any $\delta \in (H,1]$ (see, e.g., \cite[Theorem~1.6.1]{Biagini_Hu_Oksendal_Zhang}). If the H\"older--Young condition were verifiable, then there would exist exponents $\gamma,\delta>0$ such that
    \[
        F\in C^\gamma([0,1]), \qquad B^H\in C^\delta([0,1]), \qquad \gamma+\delta>1.
    \]
    Necessarily $\gamma\le\alpha$ and $\delta\le H$, and therefore $\gamma+\delta\le\alpha+H\le1$, a contradiction.
\end{proof}

\begin{remark}[On the additional energy-regularity assumption]
    The additional assumption in Theorem~\ref{thm:integral-energy-regularity},
    \[
        L_f^{(p)}(N)\le C2^{-\rho N},
    \]
    is not needed for the existence of the Faber--Schauder integral; it is a stronger decay condition on the integrand. However, it is satisfied in several of the examples above. For instance, the H\"older estimate in Proposition~\ref{prop:holder-embedding-energy} gives exponential decay whenever the integrand has a positive margin $\alpha>1/p$, and hence applies to the H\"older--Young situation of Theorem~\ref{thm:holder-young-regime}. The same type of exponential decay holds for the deterministic level-sparse class in Proposition~\ref{prop:level-sparse-fs-membership}, and therefore for Example~\ref{ex:sparse-beyond-young}, whenever $p\alpha>\eta_f$ for the chosen integrand.

    In the stochastic examples, the random level-sparse Gaussian processes in Proposition~\ref{prop:random-level-sparse-fs} also satisfy this condition almost surely whenever $p\alpha>\eta$; this is the mechanism used in the mixed example, Example~\ref{ex:mixed-sparse-fbm}. Fractional Brownian motion also satisfies the condition away from the critical index: by the estimates in the proof of Theorem~\ref{thm:fbm-energy-space}, if $B^H\in\EE^p$ with $H>1/p$, then its level energies decay exponentially up to harmless polynomial factors.

    By contrast, the full-level constructions beyond the variation Young regime in Example~\ref{ex:beyond-variation-young} and Example~\ref{ex:stochastic-beyond-variation-young} have only polynomial level-energy tails. Thus, although the basic Faber--Schauder integrals exist for these examples, they generally fall outside the additional regularity regime of Theorem~\ref{thm:integral-energy-regularity}. In this sense, Theorem~\ref{thm:integral-energy-regularity} identifies a more regular subclass of the energy-space theory, suitable for further operations such as iterated integration.
\end{remark}

\bigskip

\section{Conclusion}    \label{sec:conclusion}

Our results show that the Faber--Schauder energy framework provides a first-order pathwise integration theory driven by dyadic coefficient summability. We conclude by providing a simple consequence of the energy condition which clarifies the relation between the Faber--Schauder integral and F\"ollmer-type pathwise integration \cite{follmer1981}.

\begin{proposition}[Zero dyadic variation and covariation]  \label{prop:zero-dyadic-variation-covariation}
    Let $p>1$ and $f\in\EE^p$. Then, for every $r\ge p$, the dyadic $r$-th variation defined in \eqref{p-th-variation} satisfies
    \[
        \sup_{t\in[0,1]}[f]^{(r)}_{\T_n}(t)\xrightarrow{n\to\infty}0.
    \]
    In particular, $f$ has vanishing $r$-th variation along the dyadic partition sequence for every $r\ge p$. Moreover, if $p,q>1$ are conjugates, $f\in\EE^p$ and $g\in\EE^q$, then
    \[
        \sup_{0\le j\le2^n} \bigg| \sum_{k=0}^{j-1} \big(f(t_{n,k+1})-f(t_{n,k})\big) \big(g(t_{n,k+1})-g(t_{n,k})\big) \bigg| \xrightarrow{n\to\infty}0.
    \]
\end{proposition}

\begin{proof}
    Writing
    \[
        \Delta_{n,k}f:=f(t_{n,k+1})-f(t_{n,k})=f_{n+1,L}(k)+f_{n+1,R}(k), \qquad n \ge 0, \quad k\in I_n,
    \]
    Lemma~\ref{lem:child-estimate} and Minkowski's inequality give
    \[
        \bigg( \sum_{k\in I_n}\big|\Delta_{n,k}f\big|^p \bigg)^{\frac1p} \le \mathcal A_f^{(p)}(n).
    \]
    Since $f\in\EE^p$, we have $\mathcal A_f^{(p)}\in\ell^p(\N_0)$, and hence $\mathcal A_f^{(p)}(n)\to0$.

    If $r\ge p$, then
    \[
        \bigg(\sum_{k\in I_n}\big|\Delta_{n,k}f\big|^r\bigg)^{\frac1r}
        \le \bigg(\sum_{k\in I_n}\big|\Delta_{n,k}f\big|^p\bigg)^{\frac1p}
        \le \mathcal A_f^{(p)}(n).
    \]
    Therefore
    \[
        [f]^{(r)}_{\T_n}(1) = \sum_{k\in I_n}\big|\Delta_{n,k}f\big|^r
        \le \big(\mathcal A_f^{(p)}(n)\big)^r \xrightarrow{n\to\infty}0.
    \]
    For general $t\in[0,1]$, the quantity $[f]^{(r)}_{\T_n}(t)$ contains all full level-$n$ increments before $t$ and at most one additional partial increment of length at most $2^{-n}$. Hence
    \[
        [f]^{(r)}_{\T_n}(t) \le [f]^{(r)}_{\T_n}(1)+\omega_f(2^{-n})^r.
    \]
    Taking the supremum over $t\in[0,1]$ and using the uniform continuity of $f$ proves the first assertion.

    For the covariation estimate, H\"older's inequality gives, for every $j=0,\dots,2^n$,
    \[
        \bigg| \sum_{k=0}^{j-1}\Delta_{n,k}f\,\Delta_{n,k}g \bigg|
        \le \bigg( \sum_{k\in I_n}\big|\Delta_{n,k}f\big|^p \bigg)^{\frac1p} \bigg( \sum_{k\in I_n}\big|\Delta_{n,k}g\big|^q \bigg)^{\frac1q}
        \le \mathcal A_f^{(p)}(n) \, \mathcal A_g^{(q)}(n).
    \]
    Since $\mathcal A_f^{(p)}(n)\to0$ and $\mathcal A_g^{(q)}(n)\to0$ as $n \to \infty$, the right-hand side tends to zero.
\end{proof}

This observation clarifies the difference between the Faber--Schauder integral developed in this paper and F\"ollmer-type pathwise calculus. F\"ollmer's theory \cite{follmer1981} is designed to retain nontrivial quadratic variation, and its higher-order variants \cite{perkowski2019} retain nontrivial $m$-th variation for even integers $m\in2\mathbb N$. In the classical F\"ollmer--It\^o formula, the integrand is of gradient type, namely it arises as a derivative of a function of the integrator. Later extensions \cite{ananova2017, CF2010} allow more general non-anticipative or path-dependent integrands, but they still retain the It\^o-type feature that nontrivial quadratic or higher-order variation produces correction terms in change-of-variable formulae.

In contrast, paths in $\EE^p$ have vanishing $p$-th variation along the dyadic partition sequence, and pairs $(f,g)\in\EE^p\times\EE^q$ with conjugate exponents have zero dyadic covariation. This is why the integration-by-parts formula of Theorem~\ref{thm:extended-integration-by-parts} has no F\"ollmer-type correction term. The present theory should therefore be viewed as a first-order, zero-covariation pathwise integration theory based on Faber--Schauder energy spaces, complementary to F\"ollmer-type pathwise It\^o calculi. Moreover, the Faber--Schauder framework treats the integrand and the integrator as two independently prescribed paths, subject to conjugate energy conditions.

A natural future direction is to investigate whether a renormalized Faber--Schauder energy theory can be developed to incorporate nontrivial dyadic variation and F\"ollmer-type correction terms. A second direction is to construct iterated Faber--Schauder integrals, building on the energy-regularity result of Section~\ref{subsec:energy-regularity-integral}. Such a development would connect the present first-order theory with the Haar-Schauder approach to rough path integration of Gubinelli, Imkeller and Perkowski \cite{GubinelliImkellerPerkowski2016}, where Haar--Schauder decompositions, paraproducts, commutator estimates and L\'evy areas are used to construct pathwise stochastic integrals beyond the Young regime. The present theory is complementary to that approach: it gives a deterministic first-order bilinear integral on $\EE^p\times \EE^q$ without prescribing a L\'evy area or a controlled-path structure. Constructing compatible iterated integrals would be a natural next step toward rough path lifts and differential equations driven by paths in Faber--Schauder energy spaces.

\bigskip

\noindent \textbf{Funding}

\medskip

\noindent The author was supported by the National Research Foundation of Korea (NRF) grant funded by the Korean government (MSIT) RS-2025-00513609.

\bigskip

\renewcommand{\bibname}{References}
\bibliography{pathwise1}
\bibliographystyle{apalike}

\end{document}